\documentclass[11pt]{amsart}
 \usepackage{geometry}                % See geometry.pdf to learn the layout options. There are lots.
\usepackage{graphicx}
\usepackage{amssymb}
\usepackage{bbm, dsfont}
\usepackage{epstopdf}
\usepackage{color}
\usepackage{comment}
\usepackage{hyperref}

\newtheorem{theorem}{Theorem}[section]
\newtheorem{lemma}[theorem]{Lemma}
\newtheorem{corollary}[theorem]{Corollary}
\newtheorem{proposition}[theorem]{Proposition}

\newtheorem{question}{Question}

\theoremstyle{definition}
\newtheorem{definition}[theorem]{Definition}
\theoremstyle{remark}
\newtheorem{remark}[theorem]{Remark}

\theoremstyle{example}

\DeclareFontFamily{OML}{rsfs}{\skewchar\font'177}
\DeclareFontShape{OML}{rsfs}{m}{n}{ <5> <6> rsfs5 <7> <8> <9>
rsfs7 <10> <10.95> <12> <14.4> <17.28> <20.74> <24.88> rsfs10 }{}
\DeclareMathAlphabet{\mathfs}{OML}{rsfs}{m}{n}

\newcommand{\BC}{{\mathbb{C}}}

\newcommand{\BH}{{\mathbb{H}}}

\newcommand{\BN}{{\mathbb{N}}}

\newcommand{\BR}{{\mathbb{R}}}

\newcommand{\BZ}{{\mathbb{Z}}}

\newcommand{\CF}{{\mathcal{F}}}

\newcommand{\CT}{{\mathcal{T}}}

\newcommand{\ind}{{\mathbbm{1}}}

\newcommand{\prob}{{\bf P}}

\newcommand{\bae}{\begin{equation}\begin{aligned}}
\newcommand{\eae}{\end{aligned}\end{equation}}
\newcommand{\ev}{\mathbf{E}}

\renewcommand{\Re}{\mathrm{Re \ }}
\renewcommand{\Im}{\mathrm{Im \ }}

\newcommand{\ignore}[1]{{}}

\newcommand{\dd}{\partial}
\newcommand{\note}[1]{{\color{red}{ \bf{ [Note: #1]}}}}

\newcommand{\HL}{\text{HL}}

\newcommand{\ima}{\text{Im}}

\newcommand{\particle}{{\bf P}}
\newcommand{\diam}{\mathrm{diam}}
\begin{document}
\title{Growth of Stationary Hastings-Levitov}

\author{Noam Berger}
\address[Noam Berger]{Technische Universit\"at M\"unchen}
\urladdr{http://www.ma.huji.ac.il/~berger/}
\email{noam.berger@tum.de}

%\author{Jacob J. Kagan}
%\address[Jacob J. Kagan]{Weizmann Institute of Science}
%%\urladdr{http://personal-homepages.mis.mpg.de/sapozh/}
%\email{jacov.kagan@gmail.com }

\author{Eviatar B. Procaccia}
\address[Eviatar B. Procaccia\footnote{Research supported by NSF grant 1812009.}]{Technion - Israel Institute of Technology and Texas A\&M University}
\urladdr{www.math.tamu.edu/~procaccia}
\email{eviatarp@technion.ac.il}

\author{Amanda  Turner}
\address[Amanda  Turner]{Lancaster University, United Kingdom}
\urladdr{http://www.maths.lancs.ac.uk/~turnera/}
\email{ a.g.turner@lancaster.ac.uk}

%\thanks{Thanks the cat in the hat}

% \affiliation{University of California Los-Angeles}

\begin{abstract}
We construct and study a stationary version of the Hastings-Levitov$(0)$ model. We prove that, unlike in the classical HL$(0)$ model, in the stationary case the size of particles attaching to the aggregate is tight, and therefore SHL$(0)$ is proposed as a potential candidate for a stationary off-lattice variant of Diffusion Limited Aggregation (DLA). The stationary setting, together with a geometric interpretation of the harmonic measure, yields new geometric results such as stabilization, finiteness of arms and arm size distribution. We show that, under appropriate scaling, arms in SHL$(0)$ converge to the graph of Brownian motion which has fractal dimension $3/2$. {Moreover we show that trees with $n$ particles reach a height of order $n^{2/3}$, corresponding to a numerical prediction of Meakin from 1983 for the gyration radius of DLA growing on a long line segment.}
\end{abstract}

\maketitle
\setcounter{tocdepth}{1}

\tableofcontents

\numberwithin{equation}{section} %\numberwithin{figure}{section}

%********************************************************************************************************************************
% ********************************** Section I - Introduction ********************************************************************
% ********************************************************************************************************************************
\section{Introduction}
In 1998 Hastings and Levitov \cite{hastings1998laplacian} proposed a family of processes defined via a composition of
conformal slit functions with slit angles sampled from the harmonic measure. This family was introduced to provide off-lattice variants of DLA (Diffusion Limited Aggregation), the Eden model and other aggregation processes.

The family of processes $\HL(\alpha)$, $\alpha\in[0,2]$, is defined as follows. Abbreviate $D_0=\{x\in\BC:|x|\le 1\}$.  Consider the conformal map $\phi^\delta$ mapping $\BC\setminus D_0$ to $\BC\setminus (D_0\cup [1,1+\delta])$, normalized so that $\phi^\delta(z)=e^cz+a+\frac{b}{z}+\cdots$, for some $c>0$, as $|z|\rightarrow\infty$ (existence and uniqueness are due to the Riemann mapping theorem). Now for a random i.i.d. sequence $\theta_k$ uniformly distributed on $\partial D_0$ and a sequence $\delta_k >0$, define $\phi_k(z)=e^{i\theta_k}\phi^{\delta_k}(e^{-i\theta_k}z)$ (so $\phi_k$ corresponds to adding a slit of size $\delta_k$ at angle $\theta_k$), and
\[
\Phi_n=\phi_1\circ\phi_2\circ\cdots\circ\phi_n
.\]
The conformal maps $\Phi_n$ define a growing sequence of compact sets $D_n$ such that $\Phi_n:\BC\setminus D_0\mapsto \BC\setminus D_n$.
Thus at $|z|\rightarrow\infty$ we can write 
\begin{equation}\label{eq:laurant}
\Phi_n(z)=e^{c_n}z+O(1),
\end{equation}
where $c_n$ is the logarithmic capacity of $D_n$, and is monotonically increasing. The normalization is chosen in a way that preserves the harmonic measure from infinity, i.e. choosing a point on $\partial D_n$ with respect to the harmonic measure is equivalent to choosing $x=\Phi_n(u_n)$, where $u_n$ is sampled uniformly on $\partial D_0$. The size of the particles we add is determined by the conformal radius, thus is of order $\delta_n|\Phi'_n(u_n)|$. The one parameter family $\HL(\alpha)$, $\alpha\in[0,2]$, is defined by choosing the size of the $n$th particle to be 
\[
\delta_{n}=\delta|\Phi'_n(u_n)|^{-\frac{\alpha}{2}}
,\]
for some fixed $\delta>0$. For $\alpha=2$ the size of the particles stays approximately constant (though the shape does not) and thus this corresponds to DLA. 

The case $\alpha=0$ is the easiest to study, since it is a composition of i.i.d.~rotations of the same conformal slit function, but the process is non-physical in the following sense: the size of added particles grows exponentially (of order $e^{c_n}$), and in fact Rohde and Zinsmeister showed in 2005 \cite{rohde2005some} that as $n$ tends to infinity the scaled process has Hausdorff dimension $1$. The $\HL(0)$ process has one especially beautiful feature. Given any interval $J$ on $\partial D_0$, let $I_n$ be the elements of $D_n\setminus (D_0\setminus J)$ in the connected set of $J$ i.e. the subset of $D_n$ emanating from $J$. Denote by $H_n(J)$ the harmonic measure of $I_n$. This quantity has an elegant geometric representation since $H_n(J)=|\Phi_n^{-1}(I_n)|$ i.e. the length of the pre-image of the set $I_n$ (see Figure \ref{fig:hastings}). Moreover, for $\alpha=0$ the attached particle sizes are not a function of the location of attachment, and the harmonic measure of $I_n$ is a martingale (see 
\cite{norris2012hastings}).
\begin{figure}
\centering
\includegraphics[height=1.6in]{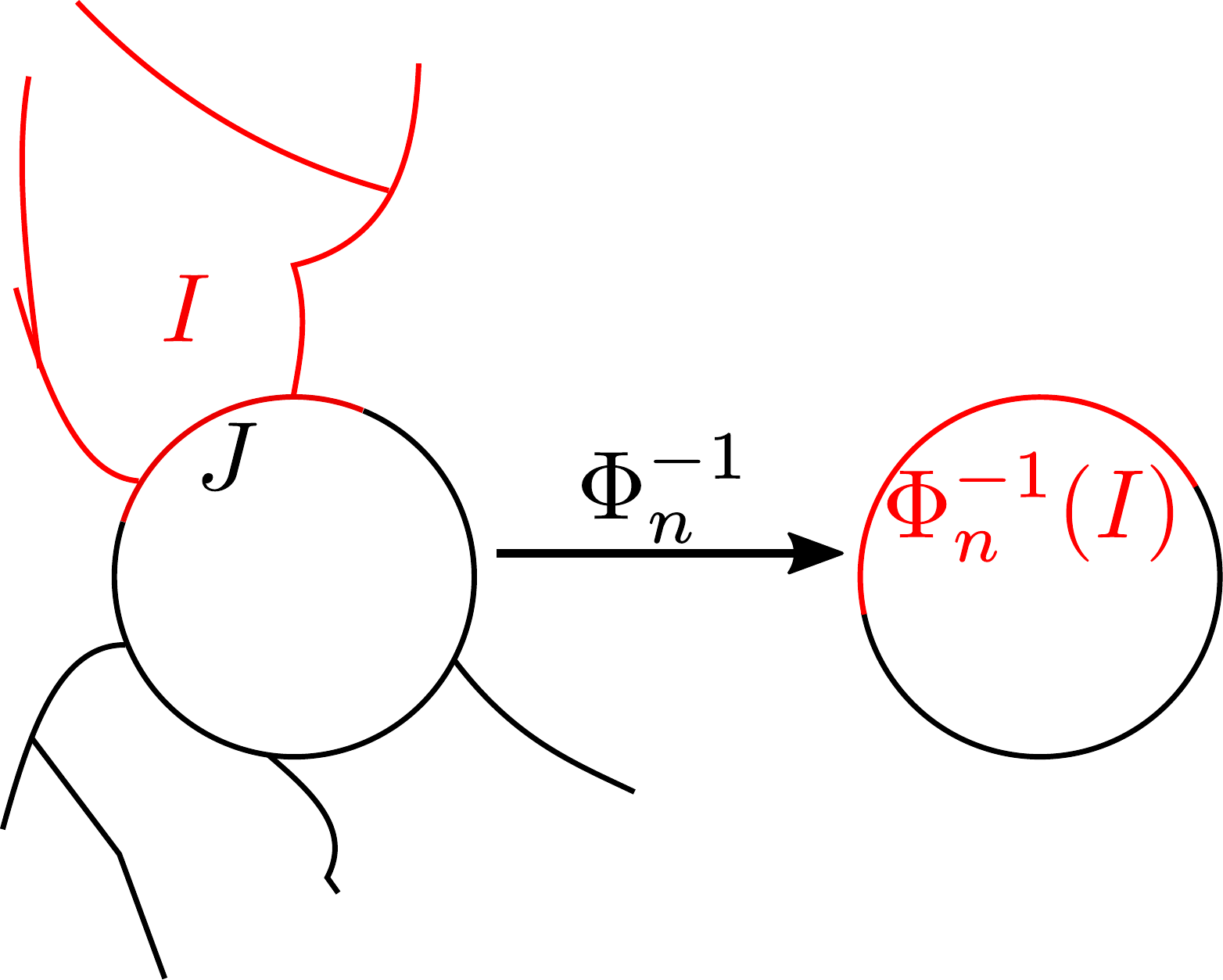}
\caption{$\HL$ geometric interpretation of the harmonic measure \label{fig:hastings}}

\end{figure}

In the 1980's, Physicists studied DLA growing from a long line segment \cite{meakin1983diffusion,racz1983diffusion,vicsek1984pattern}, a process they called Diffusion-controlled deposition on fibers. In \cite{meakin1983diffusion} Meakin predicted based on numerical simulations that the gyration radius i.e. the height of a tree with $n$ particles is $n^{2/3}$ and correspondingly the fractal dimension of DLA on a long fiber is $3/2$. In order to mathematically model random aggregation process growing from the real line, Itai Benjamini suggested considering infinite translation invariant growth processes in the upper half-plane. The IDLA and Eden models were considered by Berger, Kagan and Procaccia \cite{berger2014stretched}, and Antunovi\'c and Procaccia \cite{antunovic2014stationary} respectively. Lopez and Pimentel studied a stationary process created by geodesics in Last Passage Percolation \cite{lopez2015geodesic}, and obtained very precise geometric results. A stationary version of DLA on the lattice was introduced and studied by Procaccia et al~\cite{procaccia2018stationary, procaccia2019stationary, procaccia2017sets, procaccia2017stationary}. Recently \cite{mu2019scaling}, it was shown that DLA growing on a long line segment running up to time proportional to the length of the line, converges to stationary DLA. This justifies studying an infinite model in order to understand the phenomena of deposition on fibers considered by Meakin.
 %In all the previously studied stationary versions of aggregation processes one obtains a model with the same local behavior with the additional symmetry arising from stationarity. 
 
In this paper we prove existence and uniqueness of a Stationary Hastings Levitov process, SHL$(0)$. Moreover we prove tightness of particle sizes and that scaled appropriately the arms of the aggregate have Hausdorff dimension 3/2, corresponding to the numerical predictions of Meakin \cite{meakin1983diffusion}. 

Roughly speaking, SHL$(0)$ can be thought of as a translation-invariant growth process in the upper half-plane (abbreviated $\BH = \{z\in\BC:\ima(z)>0\}$) in which a growing cluster is represented as a composition of i.i.d.~conformal mappings. Since there is no uniform measure over $\BR$, we will use a Poisson point process to determine the location of new slits. 
Unlike the standard HL model, where particles are added one after the other, in the SHL$(0)$ the arrival times of the particles form a dense set of times which, as in many particle systems, makes well-definedness of the process non-trivial. Indeed,
a significant part of this paper is dedicated to the proof of existence (and uniqueness) of the SHL$(0)$ process. This stationary process is very special since the harmonic measure of an interval is constant in expectation and
thus the leading order of the conformal map is $z+O(1)$. This suggests that
even for $\alpha=0$, average particle sizes will be approximately constant thus yielding a stationary off-lattice variant
of DLA. One obtains all the advantages of  $\alpha=0$ (e.g. harmonic measure of an interval is a
martingale) without the undesirable exponential increase of particle size. See Figure \ref{fig:shl} for a visualization. 
\begin{figure}[h!]
\centering
\includegraphics[height=4.6cm]{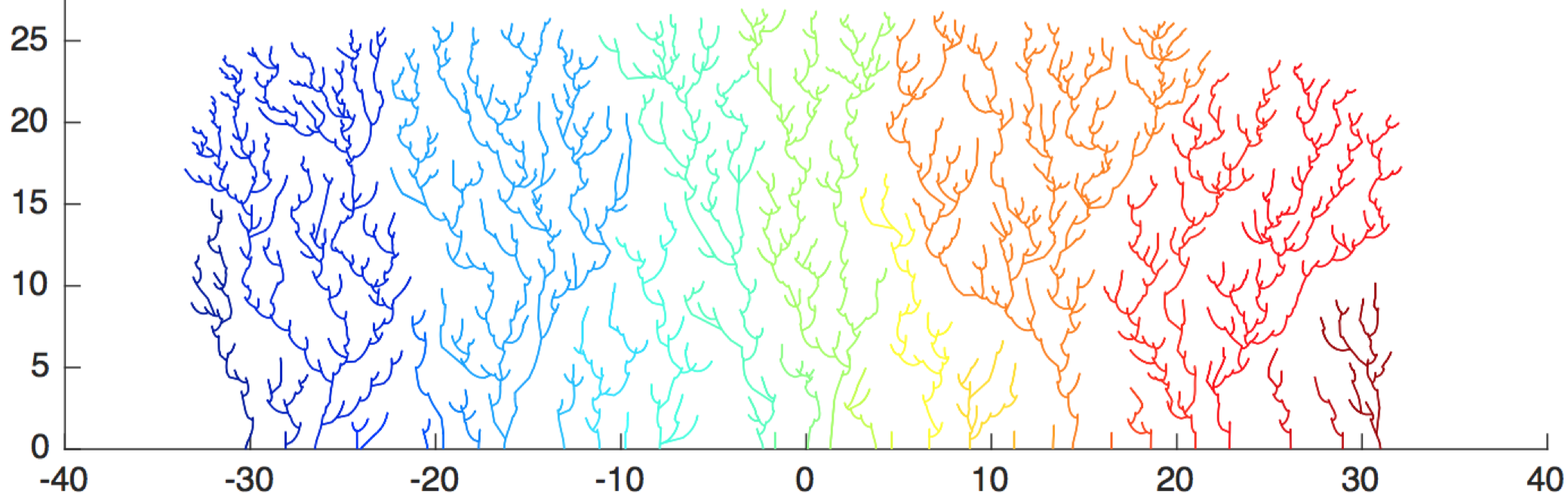}
\caption{SHL$(0)$ computer simulation. \label{fig:shl}}

\end{figure}

\subsection{Overview of results} 

The paper is structured as follows.

In Section \ref{sec:ax} we define the stationary Hastings-Levitov process SHL$(0)$ by providing the properties that the process needs to satisfy. We denote by $F_t(z)$ the conformal map from $\BH:=\{z\in\BC: \Im z>0\}$ to subsets of $\BH$ satisfying the definition of SHL$(0)$ given in Section \ref{sec:ax}. We prove the uniqueness of such a process in Section \ref{sec:unique}. In Section \ref{construction:amanda} we provide a construction of a process with these properties, thus proving existence, and use the construction to prove ergodicity of the process. 

In the remainder of the paper we establish several interesting geometric properties of SHL$(0)$. 
In Section \ref{sec:size} we prove tightness for the diameter of the particle sizes $\particle_t:=F_t(\BR)\setminus F_{t-}(\BR)$. Here, the diameter of a subset of $\BC$ is defined as the radius of the smallest Euclidean ball containing the set. This result motivates the use of SHL$(0)$ to model DLA on fibers. We also present aggregate growth bounds. In Section \ref{sec:finite_trees} we show that the process contains only finite branches with probability $1$. In Section \ref{sec:scaling} we discuss the scaling of the process and obtain the conjectured fractal dimension 3/2 of DLA on a fiber in the case of SHL$(0)$. The Hausdorff dimension on its own is not sufficient for obtaining temporal growth bounds as SHL$(0)$ is not self similar. The additional results needed to obtain the growth bounds corresponding to the numerical prediction of Meakin, namely that a tree of height $t$ contains an order of $t^{3/2}$ particles, are established in Section \ref{sec:growthrate}.

%\begin{theorem}\label{thm:prtszbnd} 
%For every $\epsilon>0$ there exists an $M<\infty$ such that for every $t$,
%\[
%P[\diam(\particle_t) > M] < \epsilon.
%\]
%\end{theorem}

%\begin{theorem}\label{thm:roughscaling}
%When appropriately scaled, the scaling limit of SHL($0$) has Hausdorff dimension $3/2$. 
%\end{theorem}

\subsection{Open problems and questions}

We finish this section with some natural directions for further research.

\begin{question}
One can try to define the SHL$(\alpha)$ process for $\alpha\neq 0$ by normalizing the slit sizes according to the conformal radius. For what values of $\alpha$ is the process well defined? We conjecture that the process is well defined for $\alpha\in(0,1)$ and not well defined for $\alpha> 2$. We have no firm conjecture for $\alpha\in[1,2]$.
\end{question}

\begin{question}
Universality of limit under different particle geometries: In \cite{norris2019scaling}, Norris, Silvestri and Turner show universality of the scaling limit for another planar aggregation model generated by composing conformal maps. We expect that the limit proved in Theorem \ref{thm:scalinglimit} for SHL$(0)$ holds for a wider class of particle geometries such as a disk. The most interesting case would be the classical off lattice DLA, where one attaches disks of equal radii according to the harmonic measure from infinity. 
\end{question}

\begin{question}
In this paper we show that the stationary Hastings Levitov process is local in the sense of strong spatial correlation decay. Show that the classical HL$(0)$ model in the small particle limit $\delta\rightarrow 0$ (see \cite{norris2012hastings}) grown up to time $t\delta^{-1}$ appropriately scaled around a point on $\BH$ and normalized such that particles are of size 1 locally converges to SHL$(0)$. In \cite{norris2012hastings}, local scaling limits are obtained by growing the HL$(0)$ process up to time $t\delta^{-2}$, so this question provides insight into the early-time behaviour of the HL$(0)$ process.  Note that this would allow one to borrow the Brownian Web scaling proved for HL$(0)$ in \cite{norris2012hastings} for the SHL$(0)$ model, rescaled as in Section \ref{sec:scaling}.
\end{question}

\begin{question}
Gaussian fluctuations of the interface: We expect from Lemma \ref{lem:prelimto_all_powerful} and the result of Silvestri  \cite{silvestri2017fluctuation} to get Gaussian fluctuations of
$$
\frac{F_t(z)-(z+i\pi t/2)}{\sqrt{t}}
,$$
as $t\rightarrow\infty$.
\end{question}

\newcommand{\R}{\mathbb R}
\newcommand{\HH}{\mathbb H}
\newcommand{\TT}{\mathcal T}
\newcommand{\SM}{\tilde s}
\newcommand{\Alpha}{\begin{picture}(8,10)\put(0,-2){\line(1,1){10}}\put(0,8){\line(1,-1){10}}\put(0,8){\line(0,-1){10}}\end{picture}}

\section{Definition of the process}\label{sec:ax}

Let $\mathcal P$ be the space of discrete measures on $[0,\infty)\times\R$, equipped with the Borel $\sigma$-algebra corresponding to the weak$^*$ topology w.r.t. continuous, compactly supported functions on $[0,\infty)\times\R$. Let $\mathcal Z$ be the space of conformal maps from $\HH$ onto subsets of $\HH$, equipped with the topology of uniform convergence on compact subsets. Let $\mathcal X$ be the space of cadlag functions from $[0,\infty)$ to $\mathcal Z$, with the Borel $\sigma$-algebra induced by the Skorohod topology (take some metric which materializes the topology in $\mathcal Z$). For $F$ in $\mathcal X$ we write $F_t$ for the function at time $t$, and $F_t(z)$ for the value that this function takes at $z$. Occasionally we will talk about $t\mapsto F_t(z)$, this is the trajectory of the point $z$ as time progresses.

We let $\Omega={\mathcal P}\times{\mathcal X}$, with the product $\sigma$-algebra $\mathcal F$.
We denote an element $\omega$ of $\Omega$ as a pair $(P,F)$ where $P\in {\mathcal P}$ and $F:[0,\infty)\to \mathcal Z$.
A measure $\mu$ on $(\Omega, \mathcal F)$ is called a SHL$(0)$ process if it satisfies a number of requirements. Before stating those requirements, we need to define the particle map by which particles are added to the cluster.

\begin{definition}\label{def:slit_map}
The {\em slit map} $\SM:\HH \to \HH\setminus [0,i]$ is defined as $\SM(z)=\sqrt{z^2 - 1}$. For $x \in \R$, the {\em slit map at $x$}, denoted $\SM_x$, is the conjugation of $\SM$ by the shift in $x$, namely 
$\SM_x(z)=x+\SM(z-x)$. 
\end{definition}

%\begin{definition}\label{def:past_danage}
%For $F\in \mathcal X$ and $z\in\HH$, we define the jump at $z$ at time $t$ to be
%\begin{equation}\label{eq:trans}
%\TT_t(z) = F_t(z) - \lim_{s \nearrow t}F_s(z).
%\end{equation}
%
%\end{definition}

\begin{definition}\label{def:SHL}
A measure $\mu$ on $(\Omega, \mathcal F)$ is called a SHL$(0)$ process if it satisfies the following requirements. 
\begin{enumerate}

\item\label{item:marg_poisson}
(Poisson arrivals) The $\mu$ marginal distribution of $P$ is that of an intensity $1$ Poisson process.

\item\label{item:start_id}
(Initial condition) $\mu$-almost surely $F_0$ is the identity.

\item\label{item:adapt} 
(Adapted) For every $0 \leq s<t$, $F_{s}^{-1} \circ F_t$ is $\CF_{s,t}$-measurable, where $\CF_{s,t} = \sigma(P|_{(s,t]\times\R})$. %\note{Is this OK now? This is OK, may be we need to explain why $F_{s}^{-1} \circ F_t$ is well-defined. May be not.}

\ignore{
\item\label{item:jump_at_poisson}
$\mu$-almost surely, for every $t\in[0,\infty)$, $F$ is continuous at $t$ if and only if $P(\{t\}\times\R)=0$.

\item\label{item:jump_is_slit}
$\mu$-almost surely, for every $(x,t) \in  [0,\infty) \times\R$, If $P(\{x,t\})=1$ (i.e if $(x,t)$ is a point of the Poisson process), then
\begin{equation}\label{eq:add_slit}
F_t = \lim_{s\nearrow t} F_s \circ \SM_x.
\end{equation}
}

\item\label{item:not_move_between_jumps}
(Growth condition)
 Let $A$ be the set of Poisson points, i.e. the atoms of $P$, and let $A_{t}=A\cap\{(s,x) : s\leq t\}$ and
 $A_{t,n}=A\cap\{(s,x):|x|\leq n\ ;\ s\leq t\}$. Then $\mu$-almost surely, for every $t\in[0,\infty)$ and $z\in\HH$,
\begin{equation}\label{eq:around_middle}
F_t(z)= z + \lim_{n\to\infty} \sum_{(s,x)\in A_{t,n}} \left[ F_{s-}(\tilde s_x(z)) - F_{s-}(z) \right].
\end{equation}

\label{item:last}
\end{enumerate}
\end{definition}

\ignore{
\subsubsection{The requirements}
A mesure $\mu$ on $(\Omega, \mathcal F)$ is called a half plane Hastings-Levitov process if it satisfies the following requirements. \note{should we name some?}
\begin{enumerate}

\item\label{item:space_shift} For $x\in\R$ and $P\in\mathcal P$ we let $\tau_x P$ be the measure satisfying $\tau_x P (A) = P(A+x)$. For $x\in\R$ and $F\in\mathcal X$ we define $\tau_x F$ as
$\tau_x F_t(z)=F_t(z+x) - x$. We require that the measure $\mu$ is $\tau_x$-invariant for every $x\in\R$.

\item\label{item:time_shift} For $t>0$ let $\mu_t$ be the distribution under $\mu$ of $\sigma_t P$ which is defined by $\sigma_t P(A)=P\big((A-it)\cap ([0,\infty)\times\R)\big)$ and of $\sigma_t F$
which is defined by $\big(\sigma_t F\big)_s = F_{t+s}\circ F_t^{-1}$. Then $\mu_t=\mu$ for every $t\geq 0$.

%\item\label{item:time_indep}
%For $t>0$, $\big(F_s(z)\big)_{z\in\HH ; s\leq t}$ is independent of $P|_{\R \times (t,\infty)}$.

\item\label{item:marg_poisson}
The $\mu$ marginal distribution of $P$ is that of an intensity $1$ Poisson process.

\item\label{item:no_prophet}
For every $T>0$, $(F_t)_{t < T}$ is independent of $P|_{ [T,\infty) \times\R}$. 

\item\label{item:start_id}
$\mu$-almost surely $F_0$ is the identity.

\item\label{item:jump_at_poisson}
$\mu$-almost surely, for every $t\in[0,\infty)$, $F$ is continuous at $t$ if and only if $P(\{t\}\times\R)=0$.

\item\label{item:jump_is_slit}
$\mu$-almost surely, for every $(x,t) \in  [0,\infty) \times\R$, If $P(\{x,t\})=1$ (i.e if $(x,t)$ is a point of the Poisson process), then
\begin{equation}\label{eq:add_slit}
F_t = \lim_{s\nearrow t} F_s \circ \SM_x.
\end{equation}

\item\label{item:not_move_between_jumps}
 Let $A$ be the set of Poisson points, i.e. the atoms of $P$, and let $A_{t}=A\cap\{(x,s) : s\leq t\}$ and
 $A_{t,y}=A\cap\{(x,s):|x|\leq y\ ;\ s\leq t\}$. Then $\mu$-almost surely, for every $t\in[0,\infty)$ and $z\in\HH$,
\begin{equation}\label{eq:around_middle}
F_t(z)= z + \lim_{y\to\infty} \sum_{(x,s)\in A_{t,y}} \TT_s(z).
\end{equation}

\label{item:last}
\end{enumerate}

Requirements \ref{item:space_shift} and \ref{item:time_shift} state some invariance properties of $\mu$. Requirement \ref{item:time_shift} would not hold for Hastings-Levitov process with parameter other than zero.
Requirement \ref{item:jump_is_slit} states that in arrival times, a conformal slit is added to the aggregate.  }
Requirement \eqref{item:not_move_between_jumps} roughly states that $F$ only changes due to jumps, where slit maps are being composed with the function. Lack of absolute summability of 
the differences forces us to choose a specific summation order, as in \eqref{eq:around_middle}. Different summation orders may lead to different processes 
%(in Section \ref{sec:pre} an explicit example is given).
(it is easy to produce such examples).
We believe that the summation order in \eqref{eq:around_middle} is the most natural one. In particular, it preserves mirror symmetry in distribution.

\subsection{Time reversal} \label{sec:reverse}
{We define here an auxiliary process, which is equal in distribution to SHL$(0)$ at any fixed time, but is more amenable to analysis.}
\begin{definition}\label{def:BSHL}
A measure $\tilde{\mu}$ on $(\Omega, \mathcal F)$, corresponding to pair $(\tilde{P}, \tilde{F})$, is called a backward SHL$(0)$ process if it satisfies the following requirements. 
\begin{enumerate}
\item\label{item:back_marg_poisson}
(Poisson arrivals) The $\tilde{\mu}$ marginal distribution of $\tilde{P}$ is that of an intensity $1$ Poisson process.

\item\label{item:back_start_id}
(Initial condition) $\tilde{\mu}$-almost surely $\tilde{F}_0$ is the identity.

\item\label{item:back_adapt} 
(Adapted) For every $0 \leq s<t$, $ \tilde{F}_t \circ \tilde{F}_{s}^{-1} $ is $\tilde \CF_{s,t}$-measurable, where $\tilde \CF_{s,t} = \sigma(\tilde P|_{(s,t]\times\R})$.

\item\label{item:back_not_move_between_jumps}
(Growth condition)
 Let $\tilde{A}$ be the set of Poisson points, i.e. the atoms of $\tilde{P}$, and let $\tilde{A}_{t}=\tilde{A}\cap\{(s,x) : s\leq t\}$ and
 $\tilde{A}_{t,n}=\tilde{A}\cap\{(s,x):|x|\leq n\ ;\ s\leq t\}$. Then $\tilde{\mu}$-almost surely, for every $t\in[0,\infty)$ and $z\in\HH$,
\begin{equation}\label{eq:back_around_middle}
\tilde{F}_t(z)= z + \lim_{n\to\infty} \sum_{(s,x)\in \tilde{A}_{t,n}} \left[ \tilde s_x(\tilde{F}_{s-}(z)) - \tilde{F}_{s-}(z) \right].
\end{equation}

\label{item:back_last}
\end{enumerate}
\end{definition}

The SHL$(0)$ and backward SHL$(0)$ processes are related as follows: Suppose that $(P,F)$ is a realization of a SHL$(0)$ process. Fix $T>0$ and define the discrete-measure $\tilde P$ on $[0,T]\times\R$ by $\tilde P(\{(t,x)\}) = P(\{(T-t,x)\})$ and the function $\tilde F: [0,T] \to \mathcal Z$ by % the cadlagization of
\[
\tilde F_t = \left[\lim_{s\searrow t}F_{T-s}\right]^{-1} \circ F_T.
\]
Then $(\tilde P, \tilde F)$ is a realization of a backward SHL$(0)$ process (restricted to $[0,T]$).

Note that the mapping that takes $(P,F)$ (with time restricted to $[0,T]$) to $(\tilde P, \tilde F)$ is a bijection since
\[
F_t = \tilde F_T \circ \left[\lim_{s\searrow t}\tilde F_{T-s}\right]^{-1}.
\]

It follows that in order to prove uniqueness and existence of the SHL$(0)$ process, it is sufficient to prove uniqueness and existence of the backward SHL$(0)$ process. 

$~$

\section{Uniqueness} \label{sec:unique} We now show that, assuming existence, there is a unique measure $\tilde \mu$ satisfying the requirements above. The existence of such a measure is proved in Section \ref{construction:amanda}.
\begin{theorem}\label{thm:uniqueness}
Let $\tilde \mu_1$ and $\tilde \mu_2$ be two probability measures satisfying Definition \ref{def:BSHL}.
Then $\tilde \mu_1 = \tilde \mu_2$.
\end{theorem}
%\ref{item:space_shift}--\ref{item:not_move_between_jumps}.

\begin{proposition}\label{prop:gunique}
Let $g:\HH \to \HH$ be a holomorphic mapping satisfying the following property.
%\item $\Im g(z) \geq 0$.
%\item %For all $z \in \HH$

For every $\xi > 0$ there exists $C(\xi) < \infty$ such that for all $y \geq \xi$
\begin{equation}\label{eq:driftg}
\lim_{n \to \infty} \left | \int_{-n}^{n} g(x + iy) dx \right | \le C(\xi),
\end{equation}
\begin{equation}\label{eq:l2g}
\int_{-\infty}^{\infty} |g(x + iy)|^2 dx \le C(\xi),
\end{equation}
%\begin{equation}\label{eq:gmonoton}
%\int_{-\infty}^\infty |g'(x+iy)| dx \leq C(\xi),
%\end{equation}
 %For all $z \in \HH$
and
\begin{equation}\label{eq:gmonoton}
\int_{-\infty}^\infty |g'(x+iy)|^2 dx \leq C(\xi).
\end{equation}

Fix $z \in \HH$.
Let $G^{(1)}, G^{(2)} :  [0, \infty) \to \HH$ be two cadlag solutions to
\begin{equation}\label{eq:axg}
G(t) = z + \lim_{n \to \infty} \sum_{(s,x)\in A_{t,n}} g(G(s-) - x)
\end{equation}
which are adapted to the filtration $\CF_{t}=\CF_{0,t}$,
where $A_{t,n}$ and $\CF_{s,t}$ are as in Definition \ref{def:SHL}.

Then $P$ -  a.s. $G^{(1)} = G^{(2)}$.
\end{proposition}

\begin{proof}

%Note that from \eqref{eq:gmonoton} we immediately get
%\begin{equation}\label{eq:l2derg}
%\int_{-\infty}^{\infty} |g'(z+x)|^2 dx < \infty
%\end{equation}

Fix $T>0$ and suppose that $t \in [0,T]$. Using \eqref{eq:axg} and the fact that $g$ maps into $\HH$, it follows that
\begin{equation}\label{eq:imincrease}
\Im G^{(i)}(t) \geq \Im z
\end{equation}
for all $t \geq 0$, for $i=1,2$.

We begin by showing that $G^{(i)}(t)$ has bounded second moment. To this end, note that%Now
\begin{eqnarray*}
G^{(i)}(t)-z &=& \lim_{n \to \infty} \sum_{(s,x)\in A_{t,n}}  g\big(G^{(i)}(s-) - x\big) \\
& = & M_t^{(i)} + \lim_{n \to \infty} \int_0^t \int_{-n}^n g\big(G^{(i)}(s) -x \big) dxds,
\end{eqnarray*}
%\note{nd made the argument of $g$ the same in both lines, up to $s$ rather than $s-$.}
where $\big(M_t^{(i)}\big)$ is a martingale by the adaptedness of $G^{(i)}$. Furthermore, using Doob's inequality,
\eqref{eq:l2g} and \eqref{eq:imincrease}
\begin{eqnarray*}
\ev\left[\sup_{s \leq t} |M_s^{(i)}|^2\right] & \leq & 4 \ev\left[|M_t^{(i)}|^2\right] \\
& = & 4 \ev \left[ 
\int_0^t \int_{-\infty}^\infty \left | g \big( G^{(i)}(s) - x \big)\right |^2 dxds 
\right] \\
& \leq & 4 C(\Im z) T,
\end{eqnarray*}
%\note{changed to $ix$}
where expectation is with respect to the Poisson random measure $P$.
Set 
\[
D_t^{(i)}=\lim_{n \to \infty} \int_0^t \int_{-n}^n g \big(G^{(i)}(s) - x \big)  dxds.
\]
Then, by \eqref{eq:driftg},
\begin{eqnarray*}
\ev\left[\sup_{s \leq t} |D_s^{(i)}|^2\right] & \leq & \ev \left[ 
\left ( \int_0^t \lim_{n \to \infty} \left | \int_{-n}^n  g(G^{(i)}(s) - x) dx \right | ds 
\right)^2 \right ] \\
& \leq & \left ( C(\Im z) t \right )^2
\end{eqnarray*}
and hence
\[
\ev\left[\sup_{s \leq t} |G_s^{(i)}|^2\right] \leq 8 C(\Im z)T + 2\left ( C(\Im z) T \right )^2 < \infty.
\]

Now write $H(t) :=  G^{(1)} (t) - G^{(2)} (t)$. Then 
\[
H(t) = M_t + D_t
\]
where $M_t=M_t^{(1)}-M_t^{(2)}$ and $D_t=D_t^{(1)}-D_t^{(2)}$. 

We now estimate the second moment of $S(t) := \sup_{s\leq t} |H(s)|$. 
In what follows we show that the expectation of $S(t)^2$ w.r.t. to the Poisson measure of $P$ is zero from which it will follow that $S(t)=0$ a.s. for all $t \leq T$. Since $T>0$ is arbitrary, the result will follow.

As before
\[
S(t)^2 \leq 2 \left[
\sup_{s \leq t} |M_s|^2 + \sup_{s \leq t} |D_s|^2
\right].
\]

%We first estimate the second moment of $\sup_{s \leq t} |M_s|$, and then that of $\sup_{s \leq t} |D_s|$.
Using Doob's inequality,
\begin{eqnarray}
\nonumber
\ev\left[\sup_{s \leq t} |M_s|^2\right] & \leq & 4 \ev\left[|M_t|^2\right] \\
& = & 4 \ev \left[ 
\int_0^t \int_{-\infty}^\infty \left| g(G^{(1)}(s) - x) -  g(G^{(2)}(s) - x) \right|^2 dxds
\right].
\label{eq:intmart}
\end{eqnarray}

We now estimate the integrand in \eqref{eq:intmart}.

For given $s$, let $\gamma_{s} : [0,1] \to \BC$ be the curve which linearly interpolates from $G^{(1)}(s)-x$ to $G^{(2)}(s)-x$.
Then

\begin{eqnarray*}
 g(G^{(1)}(s) - x) -  g(G^{(2)}(s) - x) & = & \left[ G^{(1)}(s) -  G^{(2)}(s) \right] \int_0^1 g'\big(\gamma_s(u) - x\big) du \\
 & = & H(s) \int_0^1 g'\big(\gamma_s(u) - x\big) du.
\end{eqnarray*}

By \eqref{eq:imincrease}
%As $\gamma_{s}$ linearly interpolate them,
for all $u \in [0,1]$, we have that $\Im \gamma_{s}(u) \geq \Im z$.
Therefore, by Jensen, Fubini, and \eqref{eq:gmonoton}, we get that

\begin{eqnarray}
\nonumber
\ev\left[\sup_{s \leq t} |M_s|^2\right] & \leq &  4 \ev \left[ 
\int_0^t \int_{-\infty}^\infty \int_0 ^1 \left| H(s) g'\big(\gamma_s(u) - x\big) \right|^2 dudxds 
\right]\\
\nonumber
& = &
4\ev\left[ \int_0^t \big| H(s) \big|^2 \int_0 ^1 \left[\int_{-\infty}^\infty\left|  g'\big(\gamma_s(u) + x\big) \right|^2 dx \right] duds \right] \\
\nonumber
%& \leq &
%4\int_0^t  \ev\left[ \big| H(s-) \big|^2 \right] \left[\int_{-\infty}^\infty\left|  g'\big(i\cdot \Im z + x\big) \right|^2 dx \right] ds \\
& \leq &
4C(\Im z) \int_0^t \ev\left[ \big| S(s) \big|^2 \right] ds.
\label{eq:boundsupmt}
\end{eqnarray}

By an almost identical argument we get
\begin{eqnarray}\label{eq:forsupd}
\ev\left[\sup_{s \leq t} |D_s|^2\right]
\leq
C(\Im z)^2 T \int_0^t \ev\left[ \big| S(s) \big|^2 \right] ds.
\end{eqnarray}

Combining \eqref{eq:boundsupmt} and \eqref{eq:forsupd}, we get the bound
\begin{eqnarray}\label{eq:fors}
\ev\left[|S(t)|^2\right]
\leq
\left ( 8C(\Im z)+2C(\Im z)^2 T \right ) \int_0^t \ev\left[ \big| S(s) \big|^2 \right] ds.
\end{eqnarray}
We are now done, noting that the only non-negative solution for \eqref{eq:fors} is the zero function.

\end{proof}

Theorem \ref{thm:uniqueness} follows from Proposition \ref{prop:gunique}, once we have established the appropriate basic estimates on the slit map defined in Definition \ref{def:slit_map}. These are given in the lemma below, together with some additional estimates which will be of use later. The computation of the integral in \eqref{eq:sxdrift} shows that this integral is convergent, but not absolutely convergent. This lack of absolute convergence introduces some technicalities into the proof of existence and was the reason that condition  \eqref{eq:back_around_middle} needed to specify the order of summation.

\begin{lemma}\label{lemma:slit_estimates}
There exists some absolute constant $C > 0$ such that the following hold for all $z \in \HH$.
\begin{itemize}
\item[(i)] 
\begin{equation}\label{eq:l2sx}
\int_{-\infty}^{\infty} |\SM_x(z)-z|^2 dx < \frac{C}{1+\Im z},
\end{equation}
\begin{equation}\label{eq:l1sxprime}
\int_{-\infty}^{\infty} |\SM_x'(z)-1| dx < \frac{C}{\Im z},
\end{equation}
and
\begin{equation}\label{eq:sxmonoton}
\int_{-\infty}^\infty |\SM_x'(z)-1|^2 dx < \frac{C\left(1+\log((\Im z)^{-1})\mathbbm{1}_{\{\Im z < 1\}}\right)}{(1 + (\Im z)^3)}.
\end{equation}
\item[(ii)] 
\begin{equation}\label{eq:sxdrift}
\lim_{n \to \infty} \int_{-n}^n (\SM_x(z)-z) dx = i\pi/2.
\end{equation}
\item[(iii)] 
\begin{equation}\label{eq:n0sx}
\left | \int_{-n}^{n} \left [ \SM_x(z)-z \right ] dx \right | < C \log n.
\end{equation}
If $n \geq 4 \vee 2 |\Re z|$, then 
\begin{equation}\label{eq:n1sx}
\left | \int_{-n}^{n} \left [ \SM_x(z)-z \right ] dx \right | < C 
\end{equation}
and
\begin{equation}\label{eq:n2sx}
\int_{|x|>n} |\SM_x(z)-z|^2 dx < \frac{C}{n}.
\end{equation}
\end{itemize}
\end{lemma}

Proof of Lemma \ref{lemma:slit_estimates} appears in Appendix \ref{app:3.3}.

\begin{proof}[Proof of Theorem \ref{thm:uniqueness}]
Let $(\tilde P^{(1)},\tilde F^{(1)})\sim \tilde \mu_1$ and  $(\tilde P^{(2)},\tilde F^{(2)})\sim \tilde \mu_2$ be two realizations on the same probability space, coupled s.t. $\tilde P^{(1)} = \tilde P^{(2)} = \tilde P$. Our purpose is to show that a.s. $\tilde F^{(1)} = \tilde F^{(2)}$.

To see this, note that for every $z\in \HH$, $G^{(i)}(t) = \tilde F^{(i)}_t(z)$ satisfies the assumption of Proposition \ref{prop:gunique} with respect to the Poisson process
$\tilde P$ and taking $g(w) = \SM(w) - w$. Therefore, a.s $\tilde F^{(1)}_t(z) = \tilde F^{(2)}_t(z)$. Thus, if $K \subseteq \HH$ is a countable dense subset, then a.s.
\begin{equation*}
\forall_{t \in [0,T]} \forall_{z \in K} \tilde F^{(1)}_t(z) = \tilde F^{(2)}_t(z).
\end{equation*}
The fact that, for all $t \geq 0$, the functions  $\tilde F^{(1)}_t$ and $\tilde F^{(2)}_t$ are continuous in $z$ implies that a.s. $\tilde F^{(1)} = \tilde F^{(2)}$ as required.
\end{proof}

%Using Lemma \ref{lem:apriori_unqns} we can now start controlling the difference between $F^{(1)}$ and $F^{(2)}$.
$~$

\section{Construction of the stationary Hastings-Levitov process} \label{construction:amanda}

Consider an intensity 1 Poisson process $\tilde P$ in the half plane $\{(t,x):t>0 \}$, and let $\tilde A$ be the set of Poisson points.  
In this section, we prove the existence of a cadlag map $\tilde F:[0,\infty) \to \mathcal Z$ which satisfies conditions \eqref{item:back_start_id}
, \eqref{item:back_adapt} and \eqref{item:back_not_move_between_jumps} in Definition \ref{def:BSHL}.

Almost surely, $\tilde P$ has only finitely many atoms in any compact set and $\tilde P|_{\{t\}\times\R}$ has at most one atom. So without loss of generality we restrict to the event on which this holds. For each $n \in \BN$, there exists some cadlag $G^{(n)}:[0,\infty) \to \mathcal Z$ such that    
\[
G^{(n)}_t(z)= z + \sum_{(s,x)\in \tilde{A}_{t,n}} \left[ \tilde{s}_x(G^{(n)}_{s-}(z)) - G^{(n)}_{s-}(z) \right].
\]
Indeed, if $\tilde{A}\cap \{ (t,x): |x| \leq n \} = \{ (t_1,x_1), (t_2,x_2), \dots \}$ where $0<t_1< t_2 < \cdots $, then 
\begin{equation}\label{eq:slitcompn}
G^{(n)}_t(z) = \begin{cases}
z &\mbox{ for } 0 \leq t < t_1; \\
\tilde{s}_{x_k} \circ \cdots \circ \tilde{s}_{x_1}(z) &\mbox{ for } t_k \leq t < t_{k+1},
\end{cases}
\end{equation}
satisfies the required conditions.

\begin{lemma} \label{lemma:egbound}
For each $t \in [0,\infty)$, there exists some $0<C(t)<\infty$  such that, for each $n \in \BN$ and $z \in \HH$, 
\[
\ev \left [ \sup_{s \leq t} \left | G^{(n)}_s(z) - z \right |^2\right ] \leq C(t).
\]
%Hence for all $0<\epsilon<1/2$, $z \in \HH$ and $t \in [0,\infty)$, almost surely there exists some $N \in \BN$ such that for all $n \geq N$ 
%\[
%\sup_{s \leq t} |G^{(n)}_s(z)-z| \leq n^{(1+\epsilon)/2}. 
%\]
%\note{Do we need the last statement anywhere?}
\end{lemma}
\begin{proof}
Fix $z \in \HH$ and set
\[
T_n = \inf \{t \geq 0: |G^{(n)}_t(z)| > n^{1/2}/4 \}.
\]
Now
\[
G^{(n)}_t(z) - z = M_t^{(n)}(z) + D_t^{(n)}(z)
\]
where $\left ( M_t^{(n)}(z) \right )$ is a zero-mean martingale and
\[
D_t^{(n)}(z) = \int_0^t \int_{-n}^n \left[ \tilde{s}_x(G^{(n)}_{s}(z)) - G^{(n)}_{s}(z) \right] dx ds.
\]
Hence
\[
\ev \left [ \sup_{s \leq t} \left | G^{(n)}_s(z) - z \right |^2\right ] \leq 2 \ev \left [ \sup_{s \leq t} \left | M^{(n)}_s(z) \right |^2\right ] + 2 \ev \left [ \sup_{s \leq t} \left | D^{(n)}_s(z) \right |^2\right ].
\]
Observe that
\begin{align*}
\ev \left [ |M_t^{(n)}(z)|^2 \right ] 
&= \ev \left [ \int_0^t \int_{-n}^n \left| \tilde{s}_x(G^{(n)}_{s}(z)) - G^{(n)}_{s}(z) \right|^2 dx ds \right ] \\
&\leq \ev \left [ \int_0^t \frac{C}{1+\Im G^{(n)}_{s}(z)} ds \right ] \\
&\leq \frac{C t}{1+\Im z},
\end{align*}
where $C$ is the absolute constant in \eqref{eq:l2sx}. 

By \eqref{eq:n0sx},
\[
\left | \int_0^t \int_{-n}^n \left[ \tilde{s}_x(G^{(n)}_{s}(z)) - G^{(n)}_{s}(z) \right] dx \right | \leq C t \log n,
\] 
and if $n \geq 4$ and $s < T_n$ then by \eqref{eq:n1sx}
\[
\left | \int_{-n}^n \left[ \tilde{s}_x(G^{(n)}_{s}(z)) - G^{(n)}_{s}(z) \right] dx \right | \leq C.
\] 
Therefore
\[
\ev \left [ \sup_{s \leq t} \left | D^{(n)}_s(z) \right |^2\right ] \leq C t^2 \left ( 1 + (\log n)^2 \prob(T_n < t) \right ).
\]
Hence
\begin{equation}
\label{eq:gnest}
\ev \left [ \sup_{s \leq t} \left | G^{(n)}_s(z) - z \right |^2\right ] \leq Ct(1+t)\left ( 1 + (\log n)^2 \prob(T_n < t) \right )
\end{equation}
for some (different) absolute constant $C$. By Chebychev's inequality,  
\[
\prob(T_n < t) \leq \frac{C \left ( t(1+t)( 1 + (\log n)^2  ) + |z|^2 \right )}{ n}.
\]
The result holds by substituting this back into \eqref{eq:gnest} and taking the supremum over $n$. 
%The second result holds by Borel-Cantelli.
\end{proof}

\begin{lemma} \label{lemma:sumg} For each $z \in \HH$ and $t \in [0,\infty)$,
\[
\sup_{m \geq n} \ev \left [ \sup_{s \leq t} \left | G^{(m)}_s(z) - G^{(n)}_s(z) \right |^2 \right ] \to 0
\]
and
\[
\sup_{m \geq n} \ev \left [ \sup_{s \leq t} \left | \frac{\partial}{\partial z} \left ( G^{(m)}_s(z) - G^{(n)}_s(z) \right ) \right |^2 \right ] \to 0
\]
as $n \to \infty$.
\end{lemma}
\begin{proof}
Suppose first that $n \leq m \leq 2n$. Define $T_n$ as in the previous proof and set $T=T_n \wedge T_m$, so that
\begin{equation}
\label{eq:tbound}
\prob(T < t) \leq  \prob(T_n < t) \leq \frac{2\ev \left [ \sup_{s \leq t} \left | G^{(n)}_s(z) - z \right |^2\right ]+2|z|^2}{n} \leq \frac{C_0(t,z)}{n},
\end{equation}
where $C_0(t,z)=2(C(t)+|z|^2)$, with $C(t)$ from Lemma \ref{lemma:egbound}.
Using the notation in the previous proof, 
\begin{align*}
G^{(m)}_t(z) - G^{(n)}_t(z) = M_t^{(m)}(z)-M_t^{(n)}(z) + D_t^{(m)}(z)-D_t^{(n)}(z),
\end{align*}
so 
\begin{align*}
\ev \left [ \sup_{s \leq t} \left | G^{(m)}_s(z) - G^{(n)}_s(z) \right |^2 \right ] 
\leq & \ 2 \ev \left [ \sup_{s \leq t}\left|M_s^{(m)}(z)-M_s^{(n)}(z)\right |^2 \right ] \\
& \ + 2\ev \left [ \sup_{s \leq t}\left|D_s^{(m)}(z)-D_s^{(n)}(z)\right |^2 \right ].
\end{align*}
Decompose
\begin{align*}
& \ev \left [ \sup_{s \leq t}\left|M_s^{(m)}(z)-M_s^{(n)}(z)\right |^2  \right ] \\
\leq & \ 8\ev \left [ \int_0^t \int_{-n}^n \left| (\tilde{s}_x-\mathrm{id})(G^{(m)}_{s}(z)) - (\tilde{s}_x-\mathrm{id})(G^{(n)}_{s}(z)) \right|^2 dx ds \right ]\\
& \ + 8 \ev \left [ \int_0^t \int_{n<|x|<m} \left| \tilde{s}_x(G^{(m)}_{s}(z)) - G^{(m)}_{s}(z) \right|^2 dx ds \right ]. 
\end{align*}
Now, for $s<t$, 
\begin{align*}
& \ev \left [ \int_{n<|x|<m} \left| \tilde{s}_x(G^{(m)}_{s}(z)) - G^{(m)}_{s}(z) \right|^2 dx ds \right ] \\
\leq & \ \ev \left [ \int_{|x|>n} \left| \tilde{s}_x(G^{(m)}_{s}(z)) - G^{(m)}_{s}(z) \right|^2 dx ds, T > t \right ] \\
& \ + \ev \left [ \int_{-\infty}^{\infty} \left| \tilde{s}_x(G^{(m)}_{s}(z)) - G^{(m)}_{s}(z) \right|^2 dx ds, T < t \right ] \\
&\leq \frac{C_1(t,z)}{n},
\end{align*}
for some $0<C_1(t,z)<\infty$. Here, the last line used \eqref{eq:n2sx}, \eqref{eq:l2sx} and \eqref{eq:tbound}.
Also
\[
\left| (\tilde{s}_x-\mathrm{id})(G^{(m)}_{s}(z)) - (\tilde{s}_x-\mathrm{id})(G^{(n)}_{s}(z)) \right| \leq \left | G^{(m)}_{s}(z) - G^{(n)}_{s}(z) \right |
 \int_0^1 \left| \tilde{s}_x'(\gamma(u)) - 1 \right| du, 
 \]
where $\gamma(u)$ is the curve linearly interpolating between $G^{(n)}_{s}(z)$ and $G^{(m)}_{s}(z)$.
Hence, using Jensen, Fubini, \eqref{eq:sxmonoton}, and
setting
\[
g(t) = \ev \left [ \sup_{s \leq t}\left|G_s^{(m)}(z)-G_s^{(n)}(z)\right |^2 \right ],
\]
we have
\begin{align*}
\ev \left [ \sup_{s \leq t}\left|M_s^{(m)}(z)-M_s^{(n)}(z)\right |^2 ; t< T \right ] \leq  C_2(t,z) \left ( \int_0^t g(s) ds+ \frac{1 }{n} \right ).
\end{align*}

By the same argument as used in the proof of Lemma \ref{lemma:slit_estimates}, if $|\Re G_s^{(m)}(z)| < \sqrt{n}$ then
\[
\left| \int_{n<|x|<m}  \left [ \tilde{s}_x(G^{(m)}_{s}(z)) - G^{(m)}_{s}(z)  \right ] dx \right|^2 < \frac{C}{n}.
\]
Hence, a similar argument to that above gives
\[
\ev \left [ \sup_{s \leq t}\left|D_s^{(m)}(z)-D_s^{(n)}(z)\right |^2 ; t< T \right ] \leq  C_3(t,z) \left ( \int_0^t g(s) ds+ \frac{(\log n)^2 }{n} \right ).
\]
Therefore
\[
g(t) \leq C_4(t,z) \left ( \int_0^t g(s) ds+ \frac{(\log n)^2 }{n} \right )
\]
and so by Gronwall's Lemma, setting $C(t,z) = \sup_{s\leq t} C_4(s,z) \exp ( t C_4(s,z))$,
\[
g(t) \leq \frac{C(t,z)(\log n)^2}{n}.
\]
It follows that
\[
\sup_{n \leq m \leq 2n} \ev \left [ \sup_{s \leq t} \left | G^{(m)}_s(z) - G^{(n)}_s(z) \right |^2 \right ] \leq \frac{C(t,z)(\log n)^2}{n}.
\]

Now drop the upper bound restriction on $m$ and let $m \geq n$ be arbitrary. Then there exists $k \in \mathbb{N}$ such that $2^k n \leq m < 2^{k+1}n$. By the triangle inequality
\begin{align*}
\ev \left [ \sup_{s \leq t} \left | G^{(m)}_s(z) - G^{(n)}_s(z) \right |^2 \right ]^{1/2}
\leq & \ \sum_{j=1}^k \ev \left [ \sup_{s \leq t} \left | G^{(2^jn)}_s(z) - G^{(2^{j-1}n)}_s(z) \right |^2 \right ]^{1/2} \\ 
& \ + \ev \left [ \sup_{s \leq t} \left | G^{(m)}_s(z) - G^{(2^kn)}_s(z) \right |^2 \right ]^{1/2} \\
\leq & \ C(t,z)^{1/2} n^{-1/2} \log n \sum_{j=0}^{\infty} j 2^{-j/2}.  
\end{align*}
The sum is finite from which the result follows.

Very similar arguments give the result for the derivative.

%\note{Need to double check details but don't think this needs to be typed up.}
\end{proof}

The following theorem shows the existence of the SHL$(0)$ process.

\begin{theorem}\label{thm:constworks}
For any compact $t \in [0, \infty)$ and $K \subset \HH$
\[
\sup_{m \geq n} \ev \left [ \sup_{s \leq t} \sup_{z \in K} \left | G^{(m)}_s(z) - G^{(n)}_s(z) \right |^2 \right ] \to 0 
\]
as $n \to \infty$ and hence the maps $G^{(n)}$ form a Cauchy sequence in the topology of mean-squared convergence on compact subsets of $\HH$ and compact time intervals. 

 The limit is a cadlag map $\tilde F:[0,\infty) \to \mathcal Z$ which satisfies conditions \eqref{item:back_start_id}
, \eqref{item:back_adapt} and \eqref{item:back_not_move_between_jumps} in Definition \ref{def:BSHL}.
\end{theorem}  

\begin{remark}
{Note that one can construct SHL$(0)$ also via composition of solutions  of the Loewner equation, growing countably many slits at a time.  A partial proof of this construction by the authors and Jacob Kagan appears in the Ph.D Thesis of Jacob Kagan \cite{jacob}. }
\end{remark}

\begin{proof}
Suppose $f: \HH \to \BC$ is any conformal mapping. Then for any $z_0 \in \HH$, define $g:\{|z|<1\} \to \BC$ by
\[
g(z) = \frac{f(z_0+rz)}{rf'(z_0)}
\] 
where $r \leq \Im z_0$. Then $g$ is a conformal mapping with $g'(0) = 1$ and hence, by standard distortion estimates
\[
|g'(z)| \leq \frac{1+|z|}{(1-|z|)^3}. 
\]
It follows that if $w \in B(z_0, r/2)$ then
\[
|f'(w)| \leq 16 |f'(z_0)|.
\]
For any compact set $K \subset \HH$, there exist $N \in \BN$, $z_1, \dots , z_N \in K$ and $0<r_i<\Im z_i / 2 \wedge 1$ for $i=1, \dots, N$ such that 
\[
K \subseteq \bigcup_{i=1}^N B(z_i, r_i).
\]
Then if $z \in K$, there exists $i$ s.t. $z \in B(z_i, r_i)$. Then
\begin{align*}
|G^{(m)}_t(z)-G^{(n)}_t(z)|
&\leq |G^{(m)}_t(z_i)-G^{(n)}_t(z_i)| + |z-z_i| \sup_{w \in B(z_i, r_i)} \left | \frac{d}{dz} (G^{(m)}_t-G^{(n)}_t)(w) \right | \\
&\leq |G^{(m)}_t(z_i)-G^{(n)}_t(z_i)| + 16 \left | \frac{\partial}{\partial z} (G^{(m)}_t-G^{(n)}_t)(z_i) \right |
\end{align*}
 Hence
 \begin{align*}
 \ev \left [ \sup_{s \leq t} \sup_{z \in K} \left | G^{(m)}_s(z) - G^{(n)}_s(z) \right |^2 \right ] 
 &\leq  2\sum_{i=1}^N  \ev \left [ \sup_{s \leq t} \left | G^{(m)}_s(z_i) - G^{(n)}_s(z_i) \right |^2 \right ] \\
 &+ 32\sum_{i=1}^N  \ev \left [ \sup_{s \leq t} \left | \frac{\partial}{\partial z} (G^{(m)}_s-G^{(n)}_s)(z_i) \right |^2 \right ].
 \end{align*}
 The result then follows from Lemma \ref{lemma:sumg}.

Checking that the limit satisfies the required conditions is straightforward.
\end{proof}

Next we discuss two basic properties, mainly Markovity and ergodicity, of SHL$(0)$ which we prove as corollaries from the construction in this section.

\begin{proposition}\label{prop:markov}
Let $\tilde{F}_t$ and $\hat{F}_t$ be two independent backwards SHL$(0)$. Then for every $t,s>0$, $\tilde{F}_{t+s}=\hat{F}_s\circ \tilde{F}_t$, where the equality is in distribution.
\end{proposition}
\begin{proof}
For any $t>0$ and the set of points occurring after time $t$, $\tilde{A}\cap\{(s,x):|x|\le n, s> t\}=\{(s_1,x_1),(s_2,x_2),\ldots\}$, abbreviate \[\hat{G}^{(n)}_{t,s}(z)= \begin{cases}
z &\mbox{ for } 0 \leq s+t < s_1; \\
\tilde{s}_{x_k} \circ \cdots \circ \tilde{s}_{x_1}(z) &\mbox{ for } s_k \leq s+t < s_{k+1}.
\end{cases}
\]
It is immediate from the definition that 
\begin{equation}\label{eq:mark}
G_{t+s}^{(n)}(z)=\hat{G}^{(n)}_{t,s}\circ G_t^{(n)}(z).
\end{equation} 
Moreover by the properties of $\tilde{A}$, $\hat{G}^{(n)}_{t,s}(z)$ and $G_t^{(n)}(z)$ are independent for every $n\in\BZ$. By using Theorem \ref{thm:constworks} three times on each of the functions in \eqref{eq:mark} we obtain the statement.
\end{proof}
Next we prove that the process is ergodic. First we show that composition of slit functions commutes with real translations. 
\begin{lemma}\label{lem:slitcomm}
For all $x_1,\ldots, x_n\in\BR$, and any $y\in\BR$
\[
\tilde{s}_{x_k} \circ \cdots \circ \tilde{s}_{x_1}(z-y)+y=\tilde{s}_{x_k+y} \circ \cdots \circ \tilde{s}_{x_1+y}(z)
.\]
\end{lemma}
\begin{proof}
For all $z \in \mathbb{H}$ we get 
$$\tilde{s}_{x_k}\circ\tilde{s}_{x_{k-1}}(z)+y=x_k+y+\tilde{s}_0(\tilde{s}_{x_{k-1}}(z)-x_k-y+y)=\tilde{s}_{x_k+y}(\tilde{s}_{x_{k-1}}(z)+y).$$
The statement now follows by substituting $\tilde{s}_{x_{k-2}} \circ \cdots \circ \tilde{s}_{x_1}(z-y)$ in place of $z$ and using induction.
\end{proof}
\begin{proposition}\label{prop:ergodic}
SHL$(0)$ is ergodic with respect to horizontal shifts.
\end{proposition}
\begin{proof}
We prove the statement by showing that the backward SHL$(0)$ is a factor of the ergodic space of intensity 1 Poisson point process on $[0,\infty)\times \BR$. We need to show that a real translation commutes with the map that takes a Poisson point process and returns a backward SHL$(0)$. Let $y\in\BR$ and define two sequences of processes. The first is a shift of $G_t^{(n)}(z)$ defined in \eqref{eq:slitcompn}, $S_t^{(n)}(z)=G_t^{(n)}(z-y)+y$. The second $R_t^{(n)}(z)$ is the same as in \eqref{eq:slitcompn} but using the shifted set of points $\{\tilde{A}+y\}\cap\{(t,x):|x|\le n\}$. By Lemma \ref{lem:slitcomm} $S_t^{(n)}$ can be represented as a composition of slit functions (as in \eqref{eq:slitcompn}) for the set $\{\tilde{A}+y\}\cap\{(t,x):-n+y\le x\le n+y\}$. These two sets of points are equal inside the centered interval $[-|n+y|,|n+y|]\cap[-|y-n|,|y-n|]$. By the same calculations that lead to the proof of Theorem \ref{thm:constworks} we get that $S_t^{(n)}(z)$ and $R_t^{(n)}(z)$ converge a.s. to the same limit and that the limit is a backward SHL$(0)$ with respect to $\tilde{A}+y$. Since  $S_t^{(n)}(z)$ converges to a shift of a backward SHL$(0)$ with respect to $\tilde{A}$ by the translation $y$, we obtain that the translation commutes with the mapping from Poisson process to backward SHL$(0)$. %\note{should we repeat claclations for this argument}
\end{proof}
$~$

% ********************************************************************************************************************************
\section{Typical height and particle size}
\label{sec:size}

In this section we establish the geometric property of the SHL$(0)$ cluster which motivate the assertion that SHL$(0)$ is a potential candidate for a stationary off-lattice variant of DLA grown from a line, namely that particles sizes are tight.

The first lemma states that a typical point on the SHL$(0)$ interface at time $t$ is at height $\pi t/2$ i.e.~on average the interface grows at a linear speed, and that other than this linear shift the map does not distort the upper half plane too much away from the boundary.
 
\begin{lemma}\label{lem:prelimto_all_powerful}
Suppose $\tilde \mu$ is a backward $\mathrm{SHL}(0)$ process. Then
\[
\ev \left [ \tilde F_t(z) \right ]= z+i \pi t/2
\]
and
\[
\ev \left [ \tilde F'_t(z) \right ]= 1.
\]

Furthermore, There exists some absolute constant $C>0$ such that, for each $z \in \HH$ and $t \in [0,\infty)$, 
\begin{equation}\label{eq:concval}
\ev \left [ \sup_{s \leq t} \left | \tilde F_s(z) - (z+i \pi s/2) \right |^2\right ] \leq \frac{C t}{1+\Im z},
\end{equation}
and
\begin{equation}\label{eq:concder}
\ev \left [ \sup_{s \leq t} \left | \tilde F'_s(z) - 1 \right |^2\right ] \leq \frac{C (1+|\log (\Im z)|) t}{1+(\Im z)^3} \exp \left ( \frac{C (1+|\log (\Im z)|) t}{1+(\Im z)^3} \right ) < \infty,
\end{equation}
where expectation is with respect to $\tilde \mu$.

%\note{These results (in particular the last one) are very weak. There ought to be a way to exploit the fact that $\Im(\tilde F_t(z))$ grows linearly with $t$ but I haven't been able to make it work.}

\end{lemma}
\begin{proof}
Fix $z \in \HH$. Then, using Lemma \ref{lemma:slit_estimates},
\begin{align*}
\tilde F_t(z) 
&= z+ M_t(z) + \lim_{n \to \infty} \int_{0}^t \int_{-n}^n  \left[ \tilde{s}_x(\tilde F_{s}(z)) - \tilde F_{s}(z) \right] dx ds \\
&= z + M_t(z) + i\pi t/2  
\end{align*}
where $\left ( M_t(z) \right )$ is a zero-mean martingale with
\begin{align*}
\ev \left [ |M_t(z)|^2 \right ] 
&= \ev \left [ \int_0^t \int_{-\infty}^\infty \left| \tilde{s}_x(\tilde F_{s}(z)) - \tilde F_{s}(z) \right|^2 dx ds \right ] \\
&\leq \ev \left [ \int_0^t \frac{C}{1+\Im \tilde F_{s}(z)} ds \right ] \\
&\leq \frac{C t}{1+\Im z}.
\end{align*}
Here $C$ is the absolute constant from Lemma \ref{lemma:slit_estimates}. The result follows by Doob's inequality.

Similarly
\[
\tilde F'_t(z) = 1+M'_t(z)
\]
where $\ev \left [ M'_t(z) \right ]=0$ and
\begin{align*}
\ev \left [ |M'_t(z)|^2 \right ] 
&= \ev \left [ \int_0^t |\tilde F'_s(z)|^2 \int_{-\infty}^\infty \left| \tilde{s}'_x(\tilde F_{s}(z)) - 1 \right|^2 dx ds \right ] \\
&\leq \ev \left [ \int_0^t \frac{C |\tilde F'_s(z)|^2(1+|\log (\Im \tilde F_{s}(z))|1_{(\Im \tilde F_{s}(z) < 1)})}{1+(\Im \tilde F_{s}(z))^3} ds \right ] \\
&\leq \frac{2C (1+|\log (\Im z)|) t}{1+(\Im z)^3} + \frac{2C (1+|\log (\Im z)|)}{1+(\Im z)^3} \int_0^t \ev \left [ \sup_{r \leq s} \left | \tilde F'_r(z) - 1 \right |^2\right ] ds.
\end{align*}
The result follows by Doob's inequality and Gronwall's Lemma.
\end{proof}
\begin{corollary}\label{cor:condbound}
Let $Z$ be independent of $\tilde{F}_1(z)$ satisfying $\Im Z>h>0$, then $$\ev[|\tilde{F}_1'(Z)|]\le 1+\frac{C\log h}{1+h^3}.$$
\end{corollary}%\label{cor:function_lln}
The second corollary that we will need is a law of large numbers behavior for $\Im \tilde{F}_t(0)$. Abbreviate the event $\mathfs{G}_T=\{\forall t>T, \Im \tilde{F}_t(0)>t/2\}.$
\begin{corollary}\label{cor:unfboundderivative}
 For every $\epsilon>0$ there is a $t_0>0$ such that $\prob(\mathfs{G}_{t_0})>1-\epsilon$.
\end{corollary}

\begin{proof} Let $t_n=n^2$. By Lemma \ref{lem:prelimto_all_powerful},
\begin{align}
\prob\left(\Im \tilde{F}_{t_n}(0)<\frac{t_n}{1.9}\right)\le \prob\left(\left|\tilde{F}_{t_n}(0)-\frac{i\pi t_n}{2}\right|^2>\frac{t_n^2}{4}\right)<\frac{C}{t_n}=\frac{C}{n^2}.
\end{align}
By Borel-Cantelli with probability $1$ there is some $N\in\BN$ such that for all $n>N$, $\Im \tilde{F}_{t_n}(0)>\frac{t_n}{1.9}$. Now for every $(n-1)^2<t<n^2$ with $n>N$, using the monotonicity in time of the imaginary part of the process we obtain
$$
\Im \tilde{F}_t(0)\ge\Im \tilde{F}_{(n-1)^2}(0) >\frac{(n-1)^2}{1.9}=\frac{n^2}{1.9}\frac{(n-1)^2}{n^2}>\frac{t}{2}
,$$
where the last inequality holds for large enough $n$. Finally the Lemma follows since $\lim_{t\rightarrow\infty}\prob(\mathfs{G}_t)=\prob(\exists t_0\forall t>t_0~\Im \tilde{F}_t(0)>t/2)=1$.
\end{proof}

\ignore{
Next we prove tightness for the particles size. 
We start with Theorem \ref{thm:derbndproc} which governs the derivative of the mapping on the boundary.
\begin{theorem}\label{thm:derbndproc}
For every $\epsilon>0$ there is a $M>0$ such that for all $t>0$ $$\prob\left(|\tilde{F}'_t(0)|>M\right)<\epsilon.$$
\end{theorem}
\begin{proof}
Let $\epsilon>0$. By Corollary \ref{cor:unfboundderivative} there is a $t_0>0$ such that $\prob(\mathfs{G}_{t_0})>1-\epsilon$.
Define a sequence of i.i.d functions $\{G_i(z)\}_{i\ge1}$ distributed like $\tilde{F}_1(z)$. Then by the Markov property (Lemma \ref{prop:markov}) the following equalities holds in distribution $\forall n>t_0$,
\begin{equation}
\tilde{F}_n(z)=G_n\circ G_{n-1}\circ\cdots\circ G_{1}(z) =G_n\circ G_{n-1}\circ\cdots\circ G_{t_0+1}\circ\tilde{F}_{t_0}(z) .
\end{equation}
Using the chain rule we can write:
\bae
\tilde{F}'_n(0)=&G_n'(G_{n-1}\circ G_{n-2}\circ\cdots\circ G_1(0))\cdot G_{n-1}'(G_{n-2}\circ\cdots\circ G_1(0))\cdots\\
\cdot &G_{n-2}'(G_{n-3}\circ\cdots\circ G_1(0))\cdots G_{t_0+1}'(\tilde{F}_{t_0}(0))\cdot\tilde{F}_{t_0}'(0)
\eae
Now if we condition on the event $\mathfs{G}_{t_0}$ by an inductive use of the law of total expectations and Corollary \ref{cor:condbound} we get
\begin{equation}
\ev\left[|G_n'(G_{n-1}\circ G_{n-2}\circ\cdots\circ G_1(0))|\cdots |{G}_{t_0+1}'(\tilde{F}_{t_0}(0))|\bigg|\mathfs{G}_{t_0}\right]\le\prod_{m=t_0}^n\left(1+\frac{C\log m}{m^3}\right),
\end{equation}
which converges to $1$ as $t_0\rightarrow\infty$ and $n\ge t_0$.
The last estimate we use which is the reason we only prove tightness and not present moment bounds is that for every $t_0,\epsilon>0$ there is a $\tilde{M}$ such that $\prob(|\tilde{F}'_{t_0}(0)|>\tilde{M})<\epsilon$. This follows from the fact that the process is well defined. Now abreviate $\tilde{\mathfs{G}}_{t_0}={\mathfs{G}}_{t_0}\cap\{\tilde{F}'_{t_0}(0)|<\tilde{M}\}$.
Since
\bae
\prob(|\tilde{F}_n'(0)|<M)\ge\prob(|\tilde{F}_n'(0)|<M|\tilde{\mathfs{G}}_{t_0})\prob(\tilde{\mathfs{G}}_{t_0})
\eae
we can given $\epsilon$ choose and $M,\tilde{M}>0$ and $t_0\in\BN$ such that both elements in the RHS are larger than $(1-\epsilon)$ yielding $\prob(|\tilde{F}_n'(0)|>M)<2\epsilon$.

\end{proof}

We now continue to implement Theorem \ref{thm:derbndproc} in order to control the actual size of the particle. Fix $t>0$ and
let $\hat P$ be the Poisson process $P$ with an extra point at $(t,0)$. Let $\hat F$ be the SHL process driven by $\hat P$. We define the particle arriving at time $t$ as the difference between the ranges of $\hat F_t$ and $\hat F_{t-}$, and denote it by $\particle_t$.
%\begin{theorem}\label{thm:prtszbnd} 
%For every $\epsilon>0$ there exists $M$ such that for every $t$,
%\[
%P[\diam(\particle_t) > M] < \epsilon.
%\]
%\end{theorem}
\note{Theorem \ref{thm:prtszbnd} is in Intro now}
%\ref{eq:concval} \ref{eq:concder}

\begin{proof}[Proof of Theorem \ref{thm:prtszbnd}]
Note that $\particle_t = \hat F_{t-} \big((0,i]\big)$ and therefore 
\[
\diam(\particle_t) \leq \int_0^1 |\hat{F}_{t-}'(ix)|dx\le \sup_{0\le x\le 1}|\hat{F}_{t-}'(ix)|
\]
Now from the well definability of the process we have that $\forall t_0>0,\forall\epsilon>0,\exists M>0$ such that $\prob(\sup_{0\le x\le 1}|\hat{F}_{t_0}'(ix)|>M)<\epsilon$. By the continuity of the process on $\mathbb{R}\cup\BH$ we have that 
$$\lim_{\eta\rightarrow 0}\{\forall t>t_0, \inf_{0\le s\le \eta}\Im \hat{F}_t(is)>t/2\}=\{\forall t>t_0, \Im \hat{F}_t(0)>t/2\}=\mathfs{G}_{t_0},$$ 
and thus for every $\epsilon>0$ there is an $\eta>0$ small enough such that $$\prob\left(\forall t>t_0, \inf_{0\le s\le \eta}\Im \hat{F}_t(is)>t/2\Big|\mathfs{G}_{t_0}\right)>1-\epsilon.$$
$$
$$

The tightness of the integral $\int_0^\eta |\hat{F}_{t-}'(ix)|dx$ follows from Theorem \ref{thm:derbndproc} by noting that for any $m>t_0$
$$
\ev\left[\sup_{0\le x\le\eta}|G_m'(G_{m-1}\circ G_{m-2}\circ\cdots\circ G_1(ix))|\Big|\forall t>t_0, \inf_{0\le s\le \eta}\Im \hat{F}_t(is)>t/2\right]\le 1+\frac{C\log m}{m^3}
,$$ by the law of total expectations together with the uniform bound in Corollary \ref{cor:unfboundderivative}. 
For the rest of the interval, we can repeat $1/\eta$ times the same distortion theorem argument explained in the proof of Theorem \ref{thm:constworks}. We get that
$$
\sup_{\eta\le s\le 1}||\hat{F}_{t-}'(is)|\le 16^{1/\eta}|\hat{F}_{t-}'(i\eta)|
.$$
Proving the tightness of $\diam(\particle_t)$.
\end{proof}

} %end ignore

Next we prove tightness for the particles size. 
We start with a simple lemma that states that the behaviour of the process is not too bad on the boundary. % which we \note{prove in the appendix? leave unproved? leave as an exercise for Jacob?}

\begin{lemma}\label{lem:bound_diff}
For every $x \in \BR$ and $t>0$, almost surely the mapping $F_t$ is differentiable in $x$.
\end{lemma}

We provide a proof of this lemma in the appendix. Here we give a very rough sketch as to why this is true.
%At this point we do not supply a complete proof of Lemma \ref{lem:bound_diff}. 

\begin{proof}[Sketch of the proof of Lemma \ref{lem:bound_diff}]
Since $F_t$ and $\tilde F_t$ have the same distribution, it suffices to prove the statement for $\tilde F_t$. Then,
\[
\tilde F_t (x) = \lim_{n \to \infty} G_n (x) 
\]
where
$
G_n (x) 
$
is the composition of $\mbox{Pois}(tn)$ slit functions, distributed evenly in the interval $[-n,n]$. Thus $G_n$ has roughly $3 \cdot \mbox{Pois}(tn)$ points of non-differentiability,
essentially distributed evenly in the interval $[-n,n]$. Thus in every interval $[a,b]$, the number of points of non-differentiability of $G_n$ in $[a,b]$ is tight in $n$, and therefore in the limit we get finitely many points of non-differentiability of $\tilde F_t (x)$ in every interval.
\end{proof}

We now state and prove a theorem %\ref{thm:derbndproc}
which governs the derivative of the mapping on the boundary.
\begin{theorem}\label{thm:derbndproc}
For every $\epsilon>0$ there is an $M>0$ such that for all $t>0$ $$\prob\left(|\tilde{F}'_t(0)|>M\right)<\epsilon.$$
\end{theorem}
\begin{proof}
Let $\epsilon>0$. By Corollary \ref{cor:unfboundderivative} there is a $t_0>0$ such that $\prob(\mathfs{G}_{t_0})>1-\epsilon$.
Define a sequence of i.i.d functions $\{G_i(z)\}_{i\ge1}$ distributed like $\tilde{F}_1(z)$. Then by the Markov property (Lemma \ref{prop:markov}) the following equalities hold in distribution $\forall n>t_0$,
\begin{equation}
\tilde{F}_n(z)=G_n\circ G_{n-1}\circ\cdots\circ G_{1}(z) =G_n\circ G_{n-1}\circ\cdots\circ G_{t_0+1}\circ\tilde{F}_{t_0}(z) .
\end{equation}
Using the chain rule we can write:
\bae
\tilde{F}'_n(0)=&G_n'(G_{n-1}\circ G_{n-2}\circ\cdots\circ G_1(0))\cdot G_{n-1}'(G_{n-2}\circ\cdots\circ G_1(0))\cdots\\
\cdot &G_{n-2}'(G_{n-3}\circ\cdots\circ G_1(0))\cdots G_{t_0+1}'(\tilde{F}_{t_0}(0))\cdot\tilde{F}_{t_0}'(0)
\eae

We now show how to use induction on $n$ to show that for every $n>t_0$,
\begin{equation}\label{eq:evsfivesix}
\ev\Bigg[\Big|G_n'(G_{n-1}\circ G_{n-2}\circ\cdots\circ G_1(0))\Big|\cdots \Big|{G}_{t_0+1}'(\tilde{F}_{t_0}(0))\Big|\Bigg|\mathfs{G}_{t_0}\Bigg]
\le \prob(\mathfs{G}_{t_0})^{-1}\prod_{m=t_0}^n\left(1+\frac{C\log m}{m^3}\right).
\end{equation}

Write $\mathfs{G}_{t_0}^{n} = \{\forall{t_0\leq t \leq n}, \Im \tilde F_t(0) > t/2 \}$. 
Also write $\alpha_n = \prob(\mathfs{G}_{t_0}^n|\mathfs{G}_{t_0}^{n-1})$ and note that 
\[
\prod \alpha_n = \prob(\mathfs{G}_{t_0}).
\]
Let $U_n = |G_n'(G_{n-1}\circ G_{n-2}\circ\cdots\circ G_1(0))|$. Then we want to control the conditional expectation of $U_{t_0}\cdot\ldots\cdot U_n$. Remember that $G_n$ is independent of $\mathfs{G}_{t_0}^{n-1}$.
Now, by Lemma \ref{lem:prelimto_all_powerful}, a.s. 
\[
\ev\left[U_n | \mathfs{G}_{t_0}^{n-1},U_1,\ldots,U_{n-1}\right] < \left(1+\frac{C\log n}{n^3}\right),
\]
and therefore,
\begin{eqnarray*}
\ev \left[
U_{t_0}\cdot U_{t_0 + 1} \cdots U_{n} | \mathfs{G}_{t_0}^{n} \right]
&\leq&  \alpha_n^{-1} \ev \left[ U_{t_0}\cdot U_{t_0 + 1} \cdots U_{n} | \mathfs{G}_{t_0}^{n-1} \right] \\
&\leq&  \alpha_n^{-1} \ev \left[ U_{t_0}\cdot U_{t_0 + 1} \cdots U_{n-1} | \mathfs{G}_{t_0}^{n-1} \right] \cdot \left(1+\frac{C\log n}{n^3}\right).
\end{eqnarray*}

Iterating the last inequality we get \eqref{eq:evsfivesix}.

%For every $n > t_0$ we define 
%\[
%\mathfs{G}_{t_0}^n=\{\forall t>T, \Im \tilde{F}_t(0)>t/2\}.
%\]

%Now if we condition on the event $\mathfs{G}_{t_0}$ by an inductive use of the law of total expectations and Corollary \ref{cor:condbound} we get
%\begin{equation}
%\ev\left[|G_n'(G_{n-1}\circ G_{n-2}\circ\cdots\circ G_1(0))|\cdots |{G}_{t_0+1}'(\tilde{F}_{t_0}(0))|\bigg|\mathfs{G}_{t_0}\right]\le\prod_{m=t_0}^n\left(1+\frac{C\log m}%{m^3}\right),
% \end{equation}

Note that $\prob(\mathfs{G}_{t_0})$ converges to $1$ as $t_0\rightarrow\infty$ and $n\ge t_0$.
The last estimate we use 
%which is the reason we only prove tightness and not present moment bounds
is that for every $t_0,\epsilon>0$ there is a $\tilde{M}$ such that $\prob(|\tilde{F}'_{t_0}(0)|>\tilde{M})<\epsilon$. This follows immediately from Lemma \ref{lem:bound_diff}.
%\note{needs a reference to Noam's proof and should be stated in this paper}. This follows from the fact that the process is well defined. Now abreviate $\tilde{\mathfs{G}}_{t_0}={\mathfs{G}}_{t_0}\cap\{\tilde{F}'_{t_0}(0)|<\tilde{M}\}$.
Since
\bae
\prob(|\tilde{F}_n'(0)|<M)\ge\prob(|\tilde{F}_n'(0)|<M|\tilde{\mathfs{G}}_{t_0})\prob(\tilde{\mathfs{G}}_{t_0}),
\eae
 given $\epsilon$, we can choose and $M,\tilde{M}>0$ and $t_0\in\BN$ such that both elements in the RHS are larger than $(1-\epsilon)$ yielding $\prob(|\tilde{F}_n'(0)|>M)<2\epsilon$.

\end{proof}

We now continue to implement Theorem \ref{thm:derbndproc} in order to control the actual size of the particle. Fix $t>0$ and
let $\hat P$ be the Poisson process $P$ with an extra point at $(t,0)$. Let $\hat F$ be the SHL process driven by $\hat P$. We define the particle arriving at time $t$ as the difference between the ranges of $\hat F_t$ and $\hat F_{t-}$, and denote it by $\particle_t$
\begin{theorem}\label{thm:prtszbnd} 
For every $\epsilon>0$ there exists $M$ such that for every $t$,
\[
\prob[\diam(\particle_t) > M] < \epsilon.
\]
\end{theorem}

%\ref{eq:concval} \ref{eq:concder}

\begin{proof}
Note that $\particle_t = \hat F_{t-} \big((0,i]\big)$ and therefore 
\[
\diam(\particle_t) \leq \int_0^1 |\hat{F}_{t-}'(ix)|dx\le \sup_{0\le x\le 1}|\hat{F}_{t-}'(ix)|.
\]
Now from the well definability of the process we have that $\forall t_0>0,\forall\epsilon>0,\exists M>0$ such that $\prob(\sup_{0\le x\le 1}|\hat{F}_{t_0}'(ix)|>M)<\epsilon$. By the continuity of the process on $\mathbb{R}\cup\BH$ we have that 
$$\lim_{\eta\rightarrow 0}\{\forall t>t_0, \inf_{0\le s\le \eta}\Im \hat{F}_t(is)>t/2\}=\{\forall t>t_0, \Im \hat{F}_t(0)>t/2\}=\mathfs{G}_{t_0},$$ 
and thus for every $\epsilon>0$ there is an $\eta>0$ small enough such that $$\prob\left(\forall t>t_0, \inf_{0\le s\le \eta}\Im \hat{F}_t(is)>t/2\Big|\mathfs{G}_{t_0}\right)>1-\epsilon.$$
$$
$$

The tightness of the integral $\int_0^\eta |\hat{F}_{t-}'(ix)|dx$ follows from Theorem \ref{thm:derbndproc} by noting that for any $m>t_0$
$$
\ev\left[\sup_{0\le x\le\eta}|G_m'(G_{m-1}\circ G_{m-2}\circ\cdots\circ G_1(ix))|\Big|\forall t>t_0, \inf_{0\le s\le \eta}\Im \hat{F}_t(is)>t/2\right]\le 1+\frac{C\log m}{m^3}
,$$ by the law of total expectations together with the uniform bound in Corollary \ref{cor:unfboundderivative}. 
For the rest of the interval, we can repeat $1/\eta$ times the same distortion theorem argument explained in the proof of Theorem \ref{thm:constworks}. We get that
$$
\sup_{\eta\le s\le 1}||\hat{F}_{t-}'(is)|\le 16^{1/\eta}|\hat{F}_{t-}'(i\eta)|
.$$
Proving the tightness of $\diam(\particle_t)$.
\end{proof}

%We now wish to estimate the average height of the process and the size of a particle attached in the average height. Since we are working in the stationary case the average height of a point in which a particle attaches is given by the half-plane capacity. 
%\begin{definition}
%Let $B_t$ be complex Brownian motion, and define $\tau=\inf\{t>0:B_t\in \tilde{F}_t(\partial\BH)\}$. Define the half-plane capacity 
%$$
%{\rm hcap}(\tilde{F}_t(\partial\BH))=\lim_{y\rightarrow\infty}\ev^{iy}[\ima(B_\tau)]
%.$$
%\end{definition}
%By the proof of Lemma \ref{lem:prelimto_all_powerful} Since $F_t(z) = z + M_t(z) + i\pi t/2$ and $M_t(z)$ is zero mean martingale which decays to $0$ as $z\rightarrow\infty$ we obtain
%\begin{corollary}
%$$
%{\rm hcap}(\tilde{F}_t(\partial\BH))=\pi t/2
%.$$
%\end{corollary}
%
%We can now state the main result about particle sizes, mainly the average size of a particle attaching at a typical height is bounded.
%\begin{corollary}
%There is a constant $c>0$ such that for all $t>0$
%$$
%\ev\left[|\tilde{F}_t'(i\cdot{\rm hcap}(\tilde{F}_t(\partial\BH)))\right]\le c
%.$$
%\end{corollary}
%\begin{proof}
%Follows by the bound in Lemma \ref{lem:prelimto_all_powerful} evaluated at $z=i\pi t/2$.
%\end{proof}

% ***************************************************************************************************************
% ********************************** Section IV No Infinite Trees ***********************************************
% ***************************************************************************************************************
$~$

\section{{SHL$(0)$ trees are finite}}\label{sec:finite_trees}

%Theorem \ref{thm:main} constructs the translation invariant Hastings-Levitov for a fixed time $t$. 
We now turn to consider geometric properties of the SHL's evolution in time. 
First we consider the harmonic measure. By the conformal structure of the process, the rate at which particles are attached to trees emanating from an interval on $\dd\HH$ is proportional to the interval's length under the inverse map $F_t^{-1}(z)$.  Define the harmonic measure at time $t>0$ of an interval $I=(a,b)\subset\partial\BH$ to be $$H_I(t):=F_t^{-1}(b)-F_t^{-1}(a)=\|F_t^{-1}(I)\|,$$
where $\|\cdot\|$ denotes the length of an interval.

We start with a result interesting by itself holding only for the SHL$(\alpha)$, $\alpha=0$ case. 
\begin{theorem}\label{thm:martingale}
Let $I\subset\partial\BH$ be an interval. Then $\{H_I(t)\}_{t\ge 0}$ is a positive martingale.
\end{theorem}

\begin{proof}
By the adaptedness property (\ref{item:adapt}) in the definition of the process and the Markov property in Proposition \ref{prop:markov}, it is enough to show that the harmonic measure of an interval does not change in expectation for some arbitrary time say $t=1$. Denote by $I_i=[i,i+1)$ the unit intervals and $H_{I_i}(1)=H_{I_i}(0)+T_{I_i}=1+T_{I_i}$. By the ergodicity of Proposition \ref{prop:ergodic} and Birkhoff's point-wise ergodic theorem
$$
\lim_{L\rightarrow\infty}\frac{1}{2L}\sum_{i=-L}^LT_{I_i}=\ev[T_{I_1}]
, \text{ a.s.}$$
Since the interval $[-L,L]$ is sent to an interval under $F_1^{-1}$, we need only to bound  the shift under $F_1^{-1}$ of the boundary points of the interval i.e. $$\sum_{i=-L}^LT_{I_i}\le |F_1^{-1}(L)-L|+|F_1^{-1}(-L)+L|.$$ By placing $z=F_s^{-1}(x)$ in Lemma \ref{lem:prelimto_all_powerful}, there is an $M<\infty$, $\forall x\in \BR,~\ev[|F_1^{-1}(x)-x|]=M$. Thus we get that $\frac{1}{2L}\sum_{i=-L}^LT_{I_i}$ converges in probability to $0$ yielding $\ev[T_{I_1}]=0$. 
\end{proof}

The following theorem shows stabilization of SHL$(0)$. %\note{only after the stabilization we can define the process at time infinity, show it is a monotonic process}
\begin{theorem}\label{thm:finitetrees}
For every interval $I\subset\partial\HH$, $\lim_{t\rightarrow \infty}H_I(t)=0$ a.s.
\end{theorem}
\begin{proof}
Fix an interval $I\subset\HH$. By Theorem \ref{thm:martingale} $H_I(t)$ is a positive martingale, therefore it converges a.s.
%\footnote{ by theorem 5.2.9 in  \cite{durrett2010probability} for example}. 
Next we rule out the possibility that $\lim_{t\to\infty}H_I(t) = C\neq 0$. Assume for the purpose of contradiction that $\lim_{t\to\infty}H_I(t) = C\neq 0$. Since $H_I$ converges a.s., for any $\epsilon >0$ there exists an a.s. finite $t_{\epsilon}>0$ such that for all $t>t_{\epsilon}$, $|H_I(t)-C| <\epsilon$. 

Consider the next time $t>t_\epsilon$ in which a particle connects to $I$ i.e. a point in the Poisson point process $(x,t)\in A$ with $x\in F^{-1}_{t^-}(I)$ and $t>t_\epsilon$ (it exists by the assumption that $C>0$). By simple calculation 
$$\inf_{y\in[-2C,2C]}\{s_y^{-1}(2C)-2C\}\ge\sqrt{(2C)^2+1}-2C=:\delta(C)>0.$$ 
Thus since for every $\epsilon>0$ small enough $\|F_{t^-}^{-1}(I)\|<2C$,
$$
\|F_t^{-1}(I)\|-\|F_{t^-}^{-1}(I)\|=\|s_x^{-1}\circ F_{t^-}^{-1}(I)\|-\|F_{t^-}^{-1}(I)\|\ge \delta(C)
.$$
Now choose an $0<\epsilon<\delta(C)$ and we obtain that $H_I(t)>C+\epsilon$ which is a contradiction to the assumption that $C\neq 0$. 

\end{proof}

The previous theorem yields that every tree has only finitely many particles ever attaching to it. First we need to set some notation. For any time $t>0$ define $K_t=F_t(\partial \HH)$. We call the connected components of $K_t\setminus\partial\HH$ the trees at time $t$ of the stationary HL. For an interval $I\subset\partial\HH$ we call the connected  component of $K_t\setminus (\partial\HH\setminus I)$ the trees emanating from $I$ at time $t$, and we denote them by $\CT_I(t)$. From the proof of Theorem \ref{thm:finitetrees} we obtain the following corollary.

\begin{corollary}\label{cor:lastparticle}
For every interval $I$ there exists a $t_I<\infty$ such that for every $t>t_I$, $F_t^{-1}(\CT_I(t))\times\{t\}\cap A=\emptyset$ a.s. In other words after time $t_I$ the interval $I$ will receive no more particles a.s. 
\end{corollary}
\begin{remark}
Note that the random time $t_I$ is not a stopping time.
\end{remark}
The last statement of this section yields a geometric finiteness result for all trees in the half plane HL and is trivial from the stabilization and the fact the SHL is a monotone process. Let $d(\CT_I(t))=\sup_{z\in\CT_I(t)}\inf_{y\in I}|z-y|$ and $d(\CT_I)=\sup_{t\ge 0}d(\CT_I(t))$. 

\begin{corollary}\label{cor:finitetreesdiam}
$d(\CT_I)<\infty$ a.s.
\end{corollary}

% ********************************************************************************************************************************
% ********************************** Section I - Introduction ********************************************************************
% ********************************************************************************************************************************
\newcommand{\wid}{\text{Wid}}
%\subsection{Moment bound}\label{sec:mom}
Next we show that the finite trees are large in expectation.
\begin{proposition}
For every interval $I$, $\ev[d(\CT_I)]=\infty$. 
\end{proposition}
\begin{proof}
It is enough to prove the theorem to the interval $I=[0,1]$. Assume that $\ev[d(I)]<\infty$. 
\begin{align}
\ev[d(\CT_I)]&=\int_{0}^\infty\prob[d(\CT_I) > x]dx=\frac{1}{2}\int_{-\infty}^\infty\prob[d(\CT_{I}) > |x|]dx\\
&\ge \frac{1}{2}\sum_{n=-\infty}^\infty \prob[d(\CT_{I+n}) > (n+1)],
\end{align}
where the last inequality follows the translation invariance of the model. By the Borel-Cantelli Lemma we get that almost surely there is some $K\in\BN$ large enough such that for all $n\in\BZ$ such that $|n|>K$ we have $d(\CT_{I+n})<(n+1)$. By Corollary \ref{cor:finitetreesdiam}, $d(\CT_{(-K,K)})<\infty$ a.s. which means that there exists an a.s.~finite constant $M<\infty$ such that $\CT_{\BR}$ does not intersect the set $\{z\in\BH: \Im(z)>M\vee K\}\cap\{z\in\BH:\Im(z)>|\Re(z)|\}$. This contradicts the average linear growth stated in Lemma \ref{lem:prelimto_all_powerful} for any $t>0$ large enough.

%This means $\CT_{(-K,K)}(t)$ has a uniformly bounded from zero harmonic measure for any $t>0$ large enough (See \cite{procaccia2017sets} for a detailed proof). This is in contradiction to Theorem \ref{thm:finitetrees}. \note{Use Lemma 3.4 to state that no empty space}
\end{proof}

The next proposition bounds the width of trees at any time $t$ in a height $\lambda$.
For any $\lambda>0$ let $\CT^\lambda_{[0,1)}(t)$ be the connected component of $\CT_{[0,1)}(t)\cap\left(\BR\times[0,\lambda]\right)$ intersecting $\partial\BH$ (there is exactly one by properties of conformal maps). Define
$$\wid_\lambda(\CT_{[0,1)}(t))=\sup\left\{|x-y|:(x,\lambda),(y,\lambda)\in\CT^\lambda_{[0,1)}(t)\right\}.$$
Note that by planarity if two points $(x,\lambda),(y,\lambda)$ are in $\CT^\lambda_{[0,1)}(t)$, then for any $i\neq0$ and any $x<\xi<y$, $(\xi,\lambda)\notin \CT^\lambda_{[i,i+1)}(t)$.

\begin{proposition}
For any $\lambda>0$ and $t>0$ 
$$\ev\left[\wid_{\lambda}(\CT_{[0,1)}(t))\right]\le1.$$
\end{proposition}
\begin{proof}
Fix a constant $\lambda>0$. Define a function $\zeta:\BZ\times\BZ\mapsto \{0,1\}$ by 
$$
\zeta(j,i)=\max\left\{|[x,y]\cap[j,j+1)|:(x,\lambda),(y,\lambda)\in\overline{\CT_{[i,i+1)}(t)\cap\BR\times[0,\lambda)}\right\}
,$$
where the $\max\emptyset=0$.
%\begin{align}
%\zeta(j,i)=\left\{
%\begin{array}{ll}
%1&\text{ if }\CT_{[i,i+1)}(t)\cap \left([j,j+1)\times\{ct\}\right)\neq\emptyset\\
%0& \text{ otherwise}
%\end{array} \right.
%.\end{align}
It is easy to see by the translation invariance of the Poisson process $A$ under $P$ that $\zeta$ is diagonally invariant. By the mass transport principle we obtain
\begin{equation}\label{eq:masstrans}
\ev\left[\sum_i\zeta(i,0)\right]=\ev\left[\sum_i\zeta(0,i)\right].
\end{equation}
The LHS of \eqref{eq:masstrans} is greater than $\ev\left[\wid_{\lambda}(\CT_{[0,1)}(t))\right]$. The RHS of \eqref{eq:masstrans} is bounded by the length of $[0,1)\times\{\lambda\}$, proving the result.
%$$
%\ev\left[\sum_{k\in\BZ}\zeta(0,k)\right]\le 2\sum_{k=1}^\infty\prob\left(\sup_{z\in[0,1)\times\{ct\}}|\tilde{F}_t(z)-(z+i\pi s/2)|>k\right)
%\le \frac{Ct}{ct},$$
%where the last inequality was due to the proof of Lemma \ref{lem:all_powerful} (noting that $\Im z$ is the same for all $z\in[0,1)\times\{ct\}$). \note{dropping the square is in the right direction right?}
\end{proof}
$~$

\section{Scaling limits}
\label{sec:scaling}

A natural variation on the SHL$(0)$ model is to attach slit particles of length $\delta$, corresponding to the slit map $\tilde{s}_x^\delta(z)=\sqrt{(z-x)^2 - \delta^2}+x = \delta \tilde{s}_{x/\delta}(z/\delta)$ at positions of the rescaled Poisson process $\delta P$ and then analyse the corresponding conformal mappings in the limit as $\delta \to 0$. Due to scale invariance of the half-plane, this is equivalent to analysing the original model with slit particles of length 1, under the rescaling of space and time given by
\begin{equation}\label{eq:scale}
F^{\delta}_{t}(z) = \delta F_{\delta^{-1} t}(\delta^{-1}z)
\end{equation}
as $\delta \to 0$.

The following result is an immediate consequence of Lemma \ref{lem:prelimto_all_powerful}, via the scaling relationship \eqref{eq:scale} and the connection between the forward and backwards SHL processes.

\begin{theorem}\label{thm:scaling}
Suppose $\mu$ is a $\mathrm{SHL}(0)$ process. Then, for any $T>0$ and $z \in \overline{\mathbb{H}}{:=\BH\cup\partial\BH}$,
\[
\sup_{t \leq \delta^{-1}T} | F^{\delta}_{t}(z) - (z + i \pi t/2)| \to 0
\]
in probability as $\delta \to 0$.
\end{theorem}
\begin{proof}
From the proof of Lemma \ref{lem:prelimto_all_powerful}, if $z \in \overline{\mathbb{H}}$, then
\begin{align*}
\tilde F^\delta_t(z) 
&= \delta \left ( z/\delta + M_{\delta^{-1}t}(z/\delta) + i \pi \delta^{-1} t /2 \right ) \\
&= z + i \pi t/2 + \delta M_{\delta^{-1}t}(z/\delta)  
\end{align*}
and
\begin{align*}
\ev \left [ \sup_{s \leq t} \Big|\delta M_{\delta^{-1}s}(z/\delta)\Big|^2 \right ] 
&\leq \frac{C \delta t}{1+ \delta^{-1} \Im z}.
\end{align*}
If $\Im z > 0$, then the result follows immediately. 

A little more care is needed if $z \in\mathbb{R}$. By translation invariance, we may assume that $z=0$. Fix $0<t< T \delta^{-1}$. By the result above,
\begin{align*}
\ev \left [ \sup_{s \leq t} |F^\delta_s(0)- i \pi s / 2|^2 \right ] 
&\leq C \delta t,
\end{align*}
so, by Chebychev,
\[
\prob \left (\Im F_t^\delta(0) < \pi t /4 \right ) \leq \frac{C \delta}{t}.
\]
Then, returning to the proof of Lemma \ref{lem:prelimto_all_powerful}, we have
\begin{align*}
\ev \left [ \sup_{s \leq \delta^{-1}T}|F^\delta_s(0)- i \pi s /2|^2 \right ] 
&\leq \delta \ev \left [ \int_0^{\delta^{-1}T} \frac{C}{1+ \delta^{-1}\Im \tilde{F}^\delta_s(0)} ds \right ] \\
&\leq \delta \int_0^t C ds + \delta \ev \left [ \int_t^{\delta^{-1}T} \frac{C }{1+ \delta^{-1}\Im \tilde{F}^\delta_t(0)} ds \right ] \\
&\leq C \delta t  + C T \delta / t \\
&\underset{\delta \to 0}{\longrightarrow} 0,
\end{align*}
where the penultimate line follows from  bounding the expectation separately on the event $\{ \Im F_t^\delta(0) \geq \pi t/2\}$ and its complement. 
\end{proof}

Harmonic measure at time $t>0$ of an interval $I=(a,b)\subset\partial\BH$ can be defined for the rescaled process by
$$H^\delta_I(t):=(F^\delta_t)^{-1}(b)-(F^\delta_t)^{-1}(a) = \delta H_{\delta^{-1}I}(\delta^{-t}t).$$

Exactly as in Theorem \ref{thm:martingale}, $H^\delta_I(t)$ is a martingale. Furthermore, as $\delta \to 0$, it converges to a continuous process, and hence (on an appropriate time-scale) converges to a  Brownian motion. The following result makes this precise.  

\begin{theorem}
\label{thm:harmbm}
As $\delta \to 0$, the process $(F^\delta_{\delta^{-1}t})^{-1}(a) \to 2 B_t(a) / \sqrt{3}$ in distribution, where $B_t(a)$ is a Brownian motion starting from $a$.
\end{theorem}
\begin{proof}
First observe that
\begin{align*}
F^{-1}_t(z) &= z + \lim_{n\to\infty} \sum_{(s,x)\in A_{t,n}} \left[ \tilde s_x^{-1}(F^{-1}_{s-}(z)) - F_{s-}^{-1}(z) \right] \\
&= z+ M_t(z) + \lim_{n \to \infty} \int_{0}^t \int_{-n}^n  \left[ \tilde{s}_x^{-1}(\tilde F_{s}^{-1}(z)) - \tilde F_{s}^{-1}(z) \right] dx ds  
\end{align*}
where $\left ( M_t(z) \right )$ is a zero-mean martingale with
\begin{align*}
\ev \left [ |M_t(z)|^2 \right ] 
&= \ev \left [ \int_0^t \int_{-\infty}^\infty \left| \tilde{s}_x^{-1}(\tilde F^{-1}_{s}(z)) - \tilde F_{s}^{-1}(z) \right|^2 dx ds \right ].
\end{align*}
Now
\[
\tilde s_x^{-1}(z) = \sqrt{(z-x)^2+1}-x.
\]
In particular, $\tilde s_x^{-1}$ (and hence $F_t^{-1}$) maps $\mathbb{R}$ into itself. Hence, if $a \in \mathbb{R}$, then 
\[
\lim_{n \to \infty} \int_{0}^t \int_{-n}^n  \left[ \tilde{s}_x^{-1}(\tilde F_{s}^{-1}(a)) - \tilde F_{s}^{-1}(a) \right] dx ds  = 0
\]
and
\[
\ev \left [ \int_0^t \int_{-\infty}^\infty \left| \tilde{s}_x^{-1}(\tilde F^{-1}_{s}(a)) - \tilde F_{s}^{-1}(a) \right|^2 dx ds \right ]
= t \int_{-\infty}^\infty \left| \tilde{s}_x^{-1}(0) \right|^2 dx = 4t/3.\]
Hence
\[
\ev \left [ |(F^\delta_{\delta^{-1}t})^{-1}(a) - a|^2 \right ] = \delta^2 \cdot 4 (\delta^{-2} t )/3 = 4t/3. 
\]
The result follows from standard tightness computations together with L\'evy's characterization of Brownian motion.
\end{proof}

\begin{remark} We can furthermore observe that there exists some constant $C$ such that 
\[
\int_{-\infty}^\infty \left| \tilde{s}_x^{-1}(a+h) - \tilde{s}_x^{-1}(a) \right|^2 dx \leq C h^2. 
\] 
Using this condition, one can show that if $a \neq b$ then the the scaling limits of $(F^\delta_{\delta^{-1}t})^{-1}(a)$ and $(F^\delta_{\delta^{-1}t})^{-1}(b)$ evolve as independent Brownian motions until they meet, at which point they coalesce. As any pair of independent Brownian motions eventually meet, this provides an alternative proof of Theorem \ref{thm:finitetrees}, that the harmonic measure carried by any interval must eventually be equal to zero. It is possible to prove the stronger result that $(F^\delta_{\delta^{-1}t})^{-1}$ converges to the Brownian web: 
it is straightforward to show that the inverse map $F_t^{-1}$, restricted to the real line, corresponds to the coalescing stochastic flow defined in \cite{ellis2011coalescing}.  (Indeed, the motivation behind the construction of the coalescing stochastic flow in \cite{ellis2011coalescing} was a conjectured connection with an HL$(0)$ cluster growing on the line). It is shown there that the coalescing stochastic flow converges to the Brownian web under an appropriate rescaling.  
\end{remark}

\subsection{Fingers}

As the cluster at time $t$, $K_t = F_t (\partial \mathbb{H})$, is a forest, for any point $z \in K_t$ there is a unique path in $K_t$ which connects $z$ to $\partial \mathbb{H}$. For $s < t$, let $K_{s,t}(z)$ denote the unique point on this path which is in $K_s$ and furthest from $\partial \mathbb{H}$ (with distance measured along the path). (This path may consist of a single point if $z \in \partial{H}$). 

For each $s<t$, let $F^{\mathbb{R}}_{s,t} : \mathbb{R} \to \mathbb{R}$ be the solution to
\[
F^{\mathbb{R}}_{s,t}(a) = a + \lim_{n\to\infty} \sum_{(u,x)\in A_{t,n} \setminus A_{s,n}} \left[ F^{\mathbb{R}}_{s,u-}(\Re \tilde s_x(a)) - F_{s,u-}^{\mathbb{R}}(a) \right].
\]
Existence and uniqueness of this process follows in exactly the same way as $F_t$. {This is a real valued process, that as the following Lemma establishes, identifies the points on $\partial\BH$ being sent to the roots of particles on a chosen tree}.

\begin{lemma} \label{lemma:finger}
For any $a \in \mathbb{R}$ and $s<t$,
\[
K_{s,t}(F_t(a)) = F_s(F^{\mathbb{R}}_{s,t}(a)).
\]
\end{lemma}
\begin{proof}
For each $n \in \mathbb{N}$, let $K_t^n = F_t^n (\partial \mathbb{H})$ and define $K^n_{s,t}$ analogously to $K_{s,t}$, but with respect to $F^n$. Also, let
\[
F^{\mathbb{R},n}_{s,t}(a) = a + \sum_{(u,x)\in A_{t,n} \setminus A_{s,n}} \left[ F^{\mathbb{R},n}_{s,u-}(\Re \tilde s_x(a)) - F_{s,u-}^{\mathbb{R},n}(a) \right].
\]
With probability 1, there is a finite number $k$ such that $A_{t,n} \setminus A_{s,n}=\{(t_1,x_1), \dots, (t_k, s_k)\}$, where $s < t_1 < \cdots < t_k \leq t$. The identity 
\[
K^n_{s,t}(F^n_t(a)) = F^n_s(F^{\mathbb{R},n}_{s,t}(a)).
\]
follows by straightforward induction on $k$, using the fact that $F^n_t(z)=F^n_{t_{k-1}}(\tilde{s}_{x_k}(z))$. Letting $n \to \infty$ gives the required result.
\end{proof}

For $\delta > 0$, define $K^{\delta}_{s,t}$ and $F^{\mathbb{R},\delta}_{s,t}(a)$ using the same scaling as in \eqref{eq:scale}.

\begin{lemma} \label{lemma:finbm} As $\delta \to 0$, $F^{\mathbb{R},\delta}_{\delta^{-1}s,\delta^{-1}t}(a) \to 2B_{t-s}(a)/\sqrt{3}$ in distribution, where $B_t(a)$ is a Brownian motion starting from $a$.
\end{lemma}
\begin{proof}
The proof uses similar ideas to those in the rest of the paper, so only an outline is provided. Define a time-reversed version of $\tilde{F}^{\mathbb{R}}$ as was done for $F$. Then
\begin{align*}
\tilde F^{\mathbb{R}}_{s,t}(a) &= a + \lim_{n\to\infty} \sum_{(u,x)\in A_{t,n}\setminus A_{s,n}} \left[ \Re \tilde s_x( \tilde F^{\mathbb{R}}_{s, u-}(a)) - \tilde F_{s,u-}^{\mathbb{R}}(a) \right] \\
&= a+ M_{s,t}(a)   
\end{align*}
where $\left ( M_{s,t}(a) \right )$ is a zero-mean martingale with
\begin{align*}
\ev \left [ M_{s,t}(a)^2 \right ] 
&= (t-s) \int_{-\infty}^\infty \left| \Re \tilde{s}_x(0)  \right|^2 dx = 4(t-s)/3 .
\end{align*}
The remainder of the proof is similar to Theorem \ref{thm:harmbm}. 
\end{proof}

Finally, we are in a position to prove that when appropriately scaled, the scaling limit of SHL($0$) has Hausdorff dimension $3/2$. This is an immediate consequence of the following result which shows that the finger connecting a typical point on $K_{t}$ to $\partial \mathbb{H}$ satisfies the same 2:1 scaling law as 1-dimensional Brownian motion. Here, the imaginary part of the finger is identified with time, and the real part with space. 

\begin{theorem} \label{thm:scalinglimit}
Fix $t>0$, and let $B_t$ be a standard Brownian motion starting from 0. As $\delta \to 0$ 
\begin{align*}
\delta^2 \Im K_{\delta^{-2}s, \delta^{-2}t}(F_{\delta^{-2}t}(0)) &\to \pi s /2 \\
\delta \Re K_{\delta^{-2}s, \delta^{-2}t}(F_{\delta^{-2}t}(0)) &\to 2B_{t-s} /\sqrt{3}, 
\end{align*}
where convergence is in distribution in the Skorokhod topology.
\end{theorem}
\begin{proof}
By Lemma \ref{lemma:finger},
\begin{align*}
K_{\delta^{-2}s,\delta^{-2}t}(F_{\delta^{-2}t}(0)) 
&= F_{\delta^{-2}s}(F^{\mathbb{R}}_{\delta^{-2}s,\delta^{-2}t}(0)) \\
&= \delta^{-1} F^\delta_{\delta^{-1}s} (F^{\mathbb{R},\delta}_{\delta^{-1}s,\delta^{-1}t}(0)) \\
&= \delta^{-1} F^{\mathbb{R},\delta}_{\delta^{-1}s,\delta^{-1}t}(0) + i \delta^{-2} \pi s/2 \\
& \ \ + \delta^{-1}  \left ( F^\delta_{\delta^{-1}s} (F^{\mathbb{R},\delta}_{\delta^{-1}s,\delta^{-1}t}(0)) - F^{\mathbb{R},\delta}_{\delta^{-1}s,\delta^{-1}t}(0) - i \delta^{-1} \pi s/2\right ).
\end{align*}
By Theorem 7.1 and the fact that 
$$
\ev\Bigg[\ev\Big[\delta M^\delta_{\delta^{-1}s} (F^{\mathbb{R},\delta}_{\delta^{-1}s,\delta^{-1}t}(0))\Big| F^{\mathbb{R},\delta}_{\delta^{-1}s,\delta^{-1}t}(0))\Big]\Bigg]\le 
\ev\Bigg[\frac{C\delta s}{1+\delta^{-1}\Im F^{\mathbb{R},\delta}_{\delta^{-1}s,\delta^{-1}t}(0))}\Bigg]\le C\delta s
,$$
we obtain that 
$$
\lim_{\delta\rightarrow0}\left ( F^\delta_{\delta^{-1}s} (F^{\mathbb{R},\delta}_{\delta^{-1}s,\delta^{-1}t}(0)) - F^{\mathbb{R},\delta}_{\delta^{-1}s,\delta^{-1}t}(0) - i \delta^{-1} \pi s/2\right )=0
,$$ in probability.

Then the result follows by taking real and imaginary parts of  $$\delta^{-1} F^{\mathbb{R},\delta}_{\delta^{-1}s,\delta^{-1}t}(0) + i \delta^{-2} \pi s/2$$ and using Lemma \ref{lemma:finbm}.
\end{proof}
$~$

\section{Number of particles in a tree} \label{sec:growthrate}
In this section we prove that a tree of height $t$ contains an order of $t^{3/2}$ particles. 

Let $a,b\in\partial\BH$ and define the interval $I_t=[F_t^{-1}(a),F_t^{-1}(b)]$. Now consider the process
$$
M_t=-\int_0^t\|I_s\|ds+\lim_{n\rightarrow\infty}\sum_{(u,x)\in A_{t,n}}\ind_{\big\{x\in [F_u^{-1}(a),F_u^{-1}(b)]\big\}}.
$$
\begin{lemma}\label{lem:numberofparticles}
$M_t$ is a mean zero martingale with respect to $\mathcal{F}_t$. 
\end{lemma}
\begin{proof}
By the independence of $A_{t+\Delta t}\setminus A_t$ and $A_t$ and since $F_t$ is adapted, $M_t$ is a Markov chain. Thus it is enough to prove that 
$$
\lim_{\Delta t\rightarrow 0}\frac{1}{\Delta t}\ev\left[M_{t+\Delta t}-M_t\right]=0
.$$
Indeed 
$$
\ev\left[M_{t+\Delta t}-M_t\right]=-\int_t^{t+\Delta t}\|I_s\|ds+\ev\left[\lim_{n\rightarrow\infty}\sum_{(u,x)\in A_{t+\Delta t,n}\setminus A_{t,n}}\ind_{\big\{x\in [F_u^{-1}(a),F_u^{-1}(b)]\big\}}\right]
.$$
First we have that 
$$
\lim_{\Delta t\rightarrow 0}\frac{1}{\Delta t}-\int_t^{t+\Delta t}\|I_s\|ds=-\|I_t\|
.$$
On the other hand for a small $\Delta t$
$$
\lim_{n\rightarrow\infty}\sum_{(u,x)\in A_{t+\Delta t,n}\setminus A_{t,n}}\ind_{\big\{x\in [F_u^{-1}(a),F_u^{-1}(b)]\big\}}
,$$
converges to a Poisson random variable with mean $\| I_t\|\Delta t$ proving the statement.
\end{proof}
In conclusion we obtain from Lemma \ref{lem:numberofparticles} that the expected number of particles hitting the trees growing from an interval $[a,b]\subset\partial \BH$ up to time $t$ equals the expectation of $\int_0^t\|I_s\|ds$ i.e. the area between the curves $\{F_s^{-1}(a)\}_{s\le t}$ and $\{F_s^{-1}(b)\}_{s\le t}$.

\begin{theorem}
For any $a<b$, there are $0<c<C<\infty$ such that for $\tilde{a}=\sqrt{t}a,~\tilde{b}=\sqrt{t}b$,
$$
\ev\left[\lim_{n\rightarrow\infty}\sum_{(u,x)\in A_{t,n}}\ind_{\big\{x\in [F_u^{-1}(\tilde{a}),F_u^{-1}(\tilde{b})]\big\}}\right]\in(c,C)\cdot t^{3/2}
.$$ 
\end{theorem}%\note{Does the b-a die in the limit or can we only take $b-a\approx \sqrt{t}$?}

\begin{proof}
By Theorem \ref{thm:harmbm} $\{\|I_{st}\|\}_{s\le 1}/\sqrt{t}$ converges in distribution to $(b-a)+4B_s(0)/\sqrt{3}$ with Dirichlet (absorbing) boundary at the origin.
Thus for a large enough $t>0$, using the estimate for the area under the graph of Brownian  excursion \cite[Theorem 1.1]{janson2007tail}, there are $c<C$ such that
\bae
\prob\left(\int_0^t\|I_s\|ds> x\right)\in (c,C)\cdot e^{-6(xt^{-3/2})^2}%\Big(1+O\left((xt^{-3/2})^{-2}\right)\Big)
.\eae
By integrating over $x$ we obtain that 
$$
\ev\left[\int_0^t\|I_s\|ds\right]\in (c,C)\cdot t^{3/2}
.$$

Using Lemma \ref{lem:numberofparticles} we obtain the statement.
\end{proof}

{Lastly we wish to show that with high probability at times proportional to $t$ there is a tree growing from $\left[\tilde{a},\tilde{b}\right]$ of height proportional to $t$.  }

\begin{proposition} Let $\epsilon(t)=o(\sqrt{t})$ then,
$$\prob\left(\exists \frac{t}{\epsilon(t)}\le s\le t:\sup_{F_s^{-1}(\tilde{a})\le x\le F_s^{-1}(\tilde{b})}\Im F_s(x)>\frac{\pi t}{6\epsilon(t)}\right)\underset{t\rightarrow\infty}{\longrightarrow} 1 .$$
\end{proposition}
\begin{proof}
Without loss of generality assume that $a<0$ and $b>0$. Partition the interval [-t/2,t/2] to $t^{3/4}$ subintervals of length $t^{1/4}$. Call the end points $\{\zeta_i\}_{i=1}^{t^{3/4}}$.

By Theorem \ref{thm:harmbm}, $\{F_s^{-1}(\tilde{a})\}_{s\le t}/\sqrt{t}$ and $\{F_s^{-1}(\tilde{b})\}_{s\le t}/\sqrt{t}$  converge in distribution to independent Brownian motions $\{B^a_s\}_{s\le 1}$ and $\{B^b_s\}_{s\le 1}$, starting from $a$ and $b$ respectively, until they intersect and then they perform the same Brownian motion. 

Thus as $t\rightarrow \infty$ the stated probability is greater than 
$$
\lim_{t\rightarrow\infty}\prob\left(\exists \frac{t}{\epsilon(t)}\le s\le t: \left\{\bigcap_{i=1}^{t^{3/4}}\Im F_s(\zeta_i)>\frac{\pi t}{6\epsilon(t)}\right\}\cap\left\{\bigcup_{i=1}^{t^{3/4}}\zeta_i/\sqrt{t}\in \Big[B_s^a,B_s^b\Big]\right\}\cap\{B_s^a\neq B_s^b\}\right)
.$$ 
Which converges to $1$ by Lemma \ref{thm:scaling} and basic properties of Brownian motion. Note that the $\epsilon(t)$ correction arises to guarantee the high probability of non intersection of $B_s^a$ and $B_s^b$ before time $s$.
\end{proof}

\section*{Acknowledgements}
{The authors would like to thank Itai Benjamini, Jacob Kagan and Gady Kozma for fruitful discussions at the beginning of this project. Amanda Turner would like to thank the University of Geneva for a visiting position in 2019/20 during which time this work was completed.}

\bibliography{ri}
\bibliographystyle{plain}

\section*{Appendices}

\appendix

\section{Proof of Lemma \ref{lemma:slit_estimates}}\label{app:3.3}
\begin{proof}
\begin{itemize}
\item[(i)]
Without loss of generality we may suppose that $z=iy$ for some $y>0$. First suppose that $y > 2$. Then
\begin{align*}
|\SM_x(iy)-iy|^2 &= \left | \sqrt{(iy-x)^2 - 1} - (iy-x) \right |^2 \\
&= (x^2 + y^2) \left | \sqrt{1-(iy-x)^{-2} } - 1 \right |^2 \\
&\leq (x^2+y^2)^{-1}.
\end{align*} 
Hence
\[
\int_{-\infty}^{\infty} |\SM_x(z)-z|^2 dx \leq \frac{\pi}{y}.
\]
Now suppose that $y \leq 2$. Then if $|x| \leq 2$, a very crude estimate yields
\[
|\SM_x(iy)-iy|^2 \leq 36,
\]
whereas if $|x|>2$, then the same argument as above gives
\[
|\SM_x(iy)-iy|^2 \leq (x^2+y^2)^{-1} \leq x^{-2}.
\]
Hence
\begin{align*}
\int_{-\infty}^{\infty} |\SM_x(z)-z|^2 dx 
\leq \int_{-2}^2 36 dx + 2\int_2^{\infty} x^{-2} dx \leq 145.
\end{align*}
The bound in \eqref{eq:l2sx} follows. 

For the bounds \eqref{eq:l1sxprime} and  \eqref{eq:sxmonoton}, note that similarly if either $|x|$ or $y>2$,
\begin{align*}
|\SM_x'(iy)-1|^2 &= \left | \frac{iy-x}{\sqrt{(iy-x)^2 - 1}} - 1 \right |^2 \\
&=  \left | (1-(iy-x)^{-2} )^{-1/2} - 1 \right |^2 \\
&\leq (x^2+y^2)^{-2}.
\end{align*} 
If $|x|, y \leq 2$, then one may write $x=\mathrm{sgn}(x)+w$ where $|w| \leq 1$. A straightforward computation gives
\[
|\SM_x'(iy)-1|^2 \leq \frac{4 \sqrt{2}}{\sqrt{w^2 + y^2}}+ 2.
\]
Hence, if $y>2$, then
\[
\int_{-\infty}^\infty |\SM_x'(z)-1| dx \leq \frac{\pi}{y}
\]
and
\[
\int_{-\infty}^\infty |\SM_x'(z)-1|^2 dx \leq \frac{\pi}{2y^3}
\]
and if $y \leq 2$, then
\[
\int_{-\infty}^\infty |\SM_x'(z)-1| dx \leq 1+4 \sqrt{2 + 8 y^{-1}} \leq \frac{7}{y}
\]
and
\[
\int_{-\infty}^\infty |\SM_x'(z)-1|^2 dx \leq 18\left(1+  \log (y^{-1}) \mathbbm{1}_{\{y<1\}}\right),
\]
where we have not attempted to optimise the absolute constants.
\item[(ii)] Fix $z \in \HH$ and consider the contour defined by the four curves $\gamma_i$, $i=1,2,3,4$ where
\begin{align*}
\gamma_1(t) &= z-t &\mbox{ for } &t \in [-n,n] \\
\gamma_2(t) &= z-n + it &\mbox{ for } &t \in [0,R] \\
\gamma_3(t) &= z+t + iR &\mbox{ for } &t \in [-n,n] \\ 
\gamma_4(t) &= z+n - it &\mbox{ for } &t \in [0,R],
\end{align*}
for some $n^{1/2} \ll R \ll n$.
Since the function $g(w) = \sqrt{w^2 - 1}-w$ is analytic within the contour, 
\[
\int_{-n}^n (\SM_x(z)-z) dx = \oint_{\gamma_1} g(w) dw = - \left ( \oint_{\gamma_2} + \oint_{\gamma_3} + \oint_{\gamma_4}\right ) g(w) dw.
\]
Exactly the same estimates as in the proof of (i) show that if $n-|Re(z)| > 2$, then 
\[
\left | \oint_{\gamma_i} g(w) dw \right | \leq \frac{\pi}{n-|\Re z|},
\]
otherwise
\[
\left | \oint_{\gamma_i} g(w) dw \right | \leq 12 + \log R \leq 12 + \log n.
\]
Also
\begin{align*}
-\oint_{\gamma_3} g(w) dw 
&= \int_{-n}^n (\SM_x(z+iR)-(z+iR)) dx \\
&= \int_{-n}^n \frac{1}{2(z+iR-x)} dx + O\left ( \frac{n}{R^3} \right )\\
&= \int_{-n}^n \frac{\bar{z}-iR -x}{2((\Re z-x)^2+(\Im z + R)^2)} dx + O\left ( \frac{n}{R^3} \right ) \\
&= \frac{1}{2} \log \frac{(\Re z+n)^2+(\Im z + R)^2}{(\Re z-n)^2+(\Im z + R)^2} \\
& \quad + i \frac{1}{2} \left ( \tan^{-1} \left ( \frac{\Re z+n}{\Im z + R} \right ) - \tan^{-1} \left ( \frac{\Re z-n}{\Im z + R} \right ) \right ) + O\left ( \frac{n}{R^3} \right ). 
\end{align*}
Hence
\[
\oint_{\gamma_i} g(w) dw \to 0
\]
as $n \to \infty$ for $i = 2$ and $4$ and so
\[
\lim_{n \to \infty} \int_{-n}^n (\SM_x(z)-z) dx = \frac{i\pi}{2}
\]
as required.

\item[(iii)]
Bounds \eqref{eq:n0sx} and \eqref{eq:n1sx} follow by exactly the same method as in (ii), but without taking $n \to \infty$. Bound \eqref{eq:n2sx} is established similarly to (i).
\end{itemize}
\end{proof}

\section{Proof of Lemma \ref{lem:bound_diff}}

%Here I am trying to provide a simple argument showing that for every $t$ and almost every $x \in \BR$, the mapping $F_t$ is differentiable in $x$. Of course, we will show it for the mapping $\tilde F_t$ which has the same distribution but is easier to work with. We use the proof of existence (Chapter 4) and analyse boundary differentiability along this proof.

As mentioned before, it suffices to prove the statement for the function $\tilde F_t$.

Fix $t$.

By Theorem \ref{thm:constworks}, 
%In Lemma 4.2 we (hopefully) prove the convergence of the sequence of functions $\big(G^{(n)}_{t}\big)$ to the function $\tilde F_t$.
$\tilde F_t$ is the limit of the sequence $\big(G^{(n)}_{t}\big)$ in the topology of mean-squared convergence on compact subset.

For every $n$, $G^{(n)}_{t}$ is a composition of finitely many slit functions of the type $\tilde s_x(z) = x + \sqrt{(z-x)^2 - 1}$. Fix $\delta$ small. The mapping $\tilde s_x$ is conformal from the half plane to (a subset of) itself. However, it is well defined and analytic (though not injective) in the domain $H_{u,\delta} = \BH \cup B_\delta(u)$, provided that $u\in\BR$ is such that  $\{x, x-1, x+1\} \cap B_{100\delta ^ {1/2}(u)} = \emptyset$. We write $\zeta = 100\delta ^ {1/2}$.% for future use.

\ignore{
Verification of statement:
assume $x = 0$, and let $u \in \BR$ be such that $B_\delta(u) \cap \{-1,0,1\} = \emptyset$.  
$\tilde s(z) = \sqrt{z^2 -1}$. In the case $|u| > 1 + 100\delta  ^ {1/2}$, we have that $z^2$ sends $B_\delta(u)$ into $B_{3|u|\delta}(u^2)$ i.e. $z^2 -1$ sends $B_\delta(u)$ into the half plane with positive real part, so $\sqrt{}$ has a branch which is well defined and $\tilde s$ is analytic on $H_{u,\delta}$. If $100\delta  ^ {1/2} < |u| < 1-100\delta  ^ {1/2}$, then $z^2$ sends $B_\delta(u)$ into $B_{3|u|\delta}(u^2) \subseteq \{z: 0 < \Re z < 1\}$. So $z^2 -1$ sends $B_\delta(u)$ into he half plane with negative real part, where again square root has a well defined branch. Thus in this case too, $\tilde s$ is analytic on $H_{u,\delta}$.
}

We now want to control the probability that the function $G^{(n)}_{t}$ is analytic in $H_{0,\delta}$. We remember that 
\[
G^{(n)}_{t} = \tilde s_{x_N} \circ \tilde s_{x_{N-1}} \circ \cdots \circ \tilde s_{x_2} \circ \tilde s_{x_1}
\]
where $x_1,\ldots x_N$ are the Poisson arrivals. Looking at the distribution of $x_1,\ldots x_N$, we get that $N \sim \mbox{Pois(2nt)}$ and $(x_k)$ are i.i.d. $\mbox{Unif}[-n,n]$.

Write $V_k = \tilde s_{x_k} \circ \tilde s_{x_{k-1}} \circ \cdots \circ \tilde s_{x_2} \circ \tilde s_{x_1}$. Then $G^{(n)}_{t} = V_N$, and we will recursively estimate the probability that $V_k$ is analytic in $H_{0,\delta}$.

Note that 
\[
\frac{1}{n - (1 + \zeta)} \int_{1+\zeta}^n |\tilde s'(u)|du = \frac{1}{{n - (1 + \zeta)}} \int_{1+\zeta}^n \tilde s'(u)du 
= \frac{\tilde s(n) - \tilde s(1 + \zeta)}{{n - (1 + \zeta)}} < \frac{n+1}{n-2},
\]
which means that the average derivative of $\tilde s_{x}$ at zero, conditioned on the event  $\{|x| \geq 1+\zeta\}$ is bounded by $\frac{n+1}{n-2}$.

Let $u_0 = 0$ and let $I_0 = [-\delta,\delta]$. Also let $u_k = V_k(0) = \tilde s_{x_k}(u_{k-1})$ and $I_k = V_k(I_0) = \tilde s_{x_k}(I_{k-1})$. 
Let $T=\inf\{k: |x_k - u_{k-1}| < 1 + 2\zeta \}$. Then,
\[
E\Big(|I_k| \Big| T > k\Big) \leq 2\delta \left(\frac{n+1}{n-2}\right)^k \leq 2\delta \exp(3k/n).
\]

We now break into three events. Let 
\begin{align*}
&A_1 = \Big\{T > N\Big\},\\
&A_2 = \Big\{T \leq N \ ; \ B_{2\zeta + |I_{T-1}|}(u_{T-1}) \cap \{x_T, x_T -1, x_T+1\} = \emptyset \Big\},\\
&A_3 = \Big\{T \leq N \ ; \ B_{2\zeta + |I_{T-1}|}(u_{T-1}) \cap \{x_T, x_T -1, x_T+1\} \neq \emptyset \Big\}.
\end{align*}

Under the event $A_1$, the function $G^{(n)}_{t}$  is analytic (though not injective) in the domain $H_{u,\delta}$. The same happens under the event $A_2$, because in under this event $V_{T-1}$ is analytic in $B_{\delta}(u)$, and then $\tilde s_{x_T}$ analytically sends $V_{T-1}(B_{\delta}(u))$ into the upper half plane, and then all further slit functions maintain analiticity.
Under the event $A_3$ analyticity is not guaranteed, thus we need to bound the probability of this event.

But
\begin{align*}
\prob(A_3) &= \sum_{k = 1}^\infty \prob(T=k)\prob(A_3 | T=k) \\
&= \sum_{k = 1}^\infty \prob(T=k)\prob(N \geq k)\prob\Big( B_{2\zeta+ |I_{k-1}|}(u_{k-1}) \cap \{x_k, x_k -1, x_k+1\} \neq \emptyset \Big| T=k \Big) \\
&\leq \sum_{k = 1}^\infty \prob(T=k)\prob(N \geq k) \cdot 3 \Bigg(2\zeta + \ev\Big(|I_{k-1}| \Big| T=k\Big)\Bigg) \\
&\leq 6\zeta + 2\delta \sum_{k = 1}^\infty \prob(k = T)\prob(N \geq k) \exp(3k/n)\\
&\leq 6\zeta + 2\delta \sup\Big\{ \prob(N \geq k) \exp(3k/n) : k\geq 1 \Big\} \\
&\leq 6\zeta + C\delta
\end{align*}

for some constant $C = C(t)$ where the last inequality follows from the fact that $N$ is a $nt$ Poisson variable. Write $\psi = 6\zeta + C\delta$ and note that $\psi$ goes to zero as $\delta$ does.

So we have established that for every $n$, the function $G^{(n)}_{t}$ is analytic in $H_{0,\delta}$ with probability at least $1 - \psi$. Thus w.p. at least $1 - \psi$ there exists a subsequence
$\big( G^{(n_k)}_{t} \big)$ s.t. for all $k$, $G^{(n_k)}_{t}$ is analytic in $H_{0,\delta}$. Also, $G^{(n_k)}_{t}$ converges to $\tilde F_t$ in $L^2$. Thus $\tilde F_t$ is differentiable at 0 w.p. at least $1 - \psi$, and as, by the choice of $\delta$, $\psi$ is arbitrarily small we get that $\tilde F_t$ is a.s. differentiable at 0.

\end{document}